\documentclass[]{siamltex1213}
\usepackage{enumerate}
\usepackage{graphicx, caption, subcaption}
\usepackage{multirow, array}
\usepackage{xfrac}
\usepackage{amsmath, amsfonts, amssymb}
\usepackage{tikz}
\usepackage{pgfplots}
\usepackage{standalone}
\newtheorem{example}{Example}
\numberwithin{example}{section}
\newtheorem{remark}{Remark}
\numberwithin{remark}{section}
\newcommand*{\Scale}[2][4]{\scalebox{#1}{$#2$}}%
\newcommand{\reals}{\mathbb{R}}

\newcommand{\half}{\Scale[0.5]{\tfrac{1}{2}}}
\newcommand{\nhalf}{\Scale[0.5]{-\tfrac{1}{2}}}

\newcommand{\set}[1]{\{#1\}}
\newcommand{\seq}[1]{\{#1\}}
\newcommand{\abs}[1]{\lvert#1\rvert}
\newcommand{\semi}[1]{\lvert#1\rvert}
\newcommand{\norm}[1]{\lVert#1\rVert}
\renewcommand{\brack}[1]{\langle#1\rangle}

\newcommand{\inner}[2]{\ensuremath{\left(#1, #2\right)}}

\newcommand{\HnormO}[1]{\lvert#1\rvert_{1, \Omega}}
\newcommand{\HnormG}[1]{\lvert#1\rvert_{1, \Gamma}}

\newcommand{\HG}{H^{1}_0\left(\Gamma\right)}

\newcommand{\open}[1]{\left(#1\right)}
\newcommand{\close}[1]{\left[#1\right]}
\newcommand{\Amat}{\ensuremath{\mathsf{A}}}
\newcommand{\Bmat}{\ensuremath{\mathsf{B}}}
\newcommand{\Mmat}{\ensuremath{\mathsf{M}}}
\newcommand{\Tmat}{\ensuremath{\mathsf{T}}}
\newcommand{\Hmat}[1]{\ensuremath{\mathsf{H}\!\left(#1\right)}}
\newcommand{\mat}[1]{\ensuremath{\mathsf{#1}}}
\newcommand{\inv}[1]{\ensuremath{{#1}^{-1}}}
\newcommand{\ninv}[1]{\ensuremath{{#1}^{\Scale[0.5]{-1}}}}
\newcommand{\transp}[1]{\ensuremath{{#1}^{\Scale[0.5]{\top}}}}
\renewcommand{\vec}[1]{\mat{#1}}
\newcommand{\spn}{\operatorname{span}\ }
\newcommand{\op}[1]{\ensuremath{\mathcal{#1}}}
\newcommand{\AU}{\ensuremath{A_U}}  
\newcommand{\AV}{\ensuremath{A_V}}  
\newcommand{\BU}{\ensuremath{B_U}}  
\newcommand{\BV}{\ensuremath{B_V}}  
\newcommand{\BBW}{\ensuremath{\mathcal{B}_{W}}}  
\newcommand{\BBQ}{\ensuremath{\mathcal{B}_{Q}}}  
\newcommand{\AAh}{\ensuremath{\mathbb{A}}}
\newcommand{\BBh}{\ensuremath{\mathbb{B}}}
\newcommand{\BBWh}{\ensuremath{\mathbb{B}_{W}}}
\newcommand{\BBQh}{\ensuremath{\mathbb{B}_{Q}}}
\makeatletter
\renewcommand{\inf}{\mathop{\@inf\vphantom{\@sup}}}
\renewcommand{\sup}{\mathop{\@sup\vphantom{\@inf}}}
\newcommand{\@inf}{\operatorname*{inf}}
\newcommand{\@sup}{\operatorname*{sup}}
\newcommand{\overbar}[1]{\mkern 1.5mu\overline{\mkern-1.5mu#1\mkern-1.5mu}\mkern 1.5mu}
\makeatother

\usepackage[yyyymmdd,hhmmss]{datetime}
\usepackage{catchfile}
\newcommand{\getvar}[2]{%
	\CatchFileEdef#1{"|kpsewhich -var-value #2"}{\endlinechar=-1}%
}
\getvar{\usrtest}{USER}

\title{Efficient preconditioners for saddle point systems with trace
constraints coupling 2D and 1D domains
\thanks{The work of Magne Nordaas, Mikael Mortensen and Kent-Andre Mardal has
been supported by a Center of Excellence grant from the Research Council of
Norway to the Center for Biomedical Computing at Simula Research Laboratory.}
} 

\author{
Miroslav Kuchta \footnotemark[2] \and
Magne Nordaas \footnotemark[3]\and
Joris C.~G. Verschaeve \footnotemark[2]\and
Mikael Mortensen \footnotemark[2] \footnotemark[3]\and
Kent-Andre Mardal \footnotemark[2] \footnotemark[3]
}

\begin{document}
\renewcommand{\thefootnote}{\fnsymbol{footnote}}
\footnotetext[2]{Department of Mathematics, Division of Mechanics, University of Oslo\\
\texttt{\{mirok, joris, mikaem\}@math.uio.no}}
\footnotetext[3]{Center for Biomedical Computing, Simula Research Laboratory, \\
\texttt{\{magneano, kent-and\}@simula.no}
}
\renewcommand{\thefootnote}{\arabic{footnote}}

\maketitle
\slugger{SISC}{xxxx}{xx}{x}{x--x}

\begin{abstract}
We study preconditioners for a model problem describing the coupling of two 
elliptic subproblems posed over domains with different topological dimension by 
a parameter dependent constraint. A pair of parameter robust and efficient preconditioners 
is proposed and analyzed. Robustness and efficiency of the preconditioners is 
demonstrated by numerical experiments. 
\end{abstract}

\begin{keywords}preconditioning, saddle-point problem, Lagrange multipliers\end{keywords}

\begin{AMS}65F08\end{AMS}

\pagestyle{myheadings}
\thispagestyle{plain}
\markboth{Preconditioning for trace constrained systems}{}

\section{Introduction}\label{sec:intro} 
This paper is concerned with preconditioning of multiphysics problems
where two subproblems of different dimensionality are coupled. 
We assume that $\Gamma$ is a sub-manifold contained within $\Omega \in \mathbb{R}^n$
and consider the following problem:  

\begin{subequations}\label{pde}
\begin{align}
\label{pde1}
  -\Delta u + \epsilon\delta_{\Gamma} p&= f &\mbox{ in } \Omega,\\
\label{pde2}
-\Delta v - p &= g &\mbox{ on } \Gamma,\\
\label{coupling}
\epsilon u - v  &=  0 &\mbox{ on } \Gamma,
\end{align}
\end{subequations}
where $\delta_{\Gamma}$ is a function with properties similar 
to the Dirac delta function as will be discussed later. To allow for a unique
solution $(u, v, p)$ the system must be equipped with suitable boundary conditions 
and we shall here, for simplicity, consider homogeneous Dirichlet boundary conditions
for $u$ and $v$ on $\partial\Omega$ and $\partial\Gamma$ respectively. 
%
We note that the unknowns $u, v$ are here the primary variables, while the
unknown $p$ should be interpreted as a Lagrange multiplier associated with the constraint
\eqref{coupling}.

The two elliptic equations that are stated on two different domains, $\Omega$ and $\Gamma$, are
coupled and therefore the restriction of $u$ to $\Gamma$ and the extension of $p$
to $\Omega$ are crucial. When the codimension of $\Gamma$ is one, the restriction 
operator is a trace operator and the extension operator is similar to the Dirac 
delta function. We note that $\epsilon \in (0,1)$ and that  the typical
scenario will be that $\epsilon \ll 1$.  We will therefore focus on methods
that  are robust in $\epsilon$. 

The problem \eqref{pde1}--\eqref{coupling} is relevant to biomedical applications 
\cite{mox, Cattaneo, coupled_cstr_via_lm, dangelo_3d_1d} where it models the 
coupling of the porous media flow inside tissue to the vascular 
bed through Starlings law. Further, problems involving coupling of the finite 
element method and the boundary element method, e.g. \cite{funken_fem_bem,
multiscale_fem_bem}, are of the form \eqref{pde}. The system is also relevant
for domain decomposition methods based on Lagrange multipliers
\cite{domain_decomposition}. Finally, in solid mechanics, the problem of plates
reinforced with ribs, cf. for example \cite[ch. 9.11]{timoshenko}, can be recast
into a related fourth order problem. 
We also note that the techniques developed here to address the constraint 
\eqref{coupling} are applicable in preconditioning fluid-structure interaction 
problems involving interactions with thin structures, e.g. filaments \cite{etienne}.

One way of deriving equations \eqref{pde} is to consider the following  
minimization problem 
\begin{equation}
\label{eq:energy}
\left. 
\begin{array}{l}
\displaystyle\int_{\Omega} (\nabla u)^2 - 2 u f \,\mathrm{d}x \\
\displaystyle\int_{\Gamma} (\nabla v)^2 - 2 v g \,\mathrm{d}s 
\end{array}
\right\} \rightarrow \mbox{min}
\end{equation}
subject to the constraint
\begin{subequations}\label{eq:min}
\begin{align}
\label{eq:constraint}
&\epsilon u - v  =  0\quad\mbox{ on }\Gamma.
\end{align}
\end{subequations}
Using the method of Lagrange multipliers, the constrained minimization problem will 
be re-cast as a saddle-point problem. The saddle-point problem
is then analyzed in terms of the Brezzi conditions~\cite{brezzi1974existence}  
and efficient solution algorithms are obtained using operator 
preconditioning~\cite{kent_ragnar}.  
A main challenge is the fact that 
the constraint \eqref{eq:constraint} necessitates the use of trace operators which 
leads to operators in fractional Sobolev spaces on $\Gamma$. 


An outline of the paper is as follows: Section \ref{sec:prelim} presents the necessary 
notation and mathematical framework needed for the analysis. Then the mathemathical 
analysis as well as the numerical experiments of two different preconditioners are 
presented in \S \ref{sec:Qcap} and \S \ref{sec:Vcap}, respectively.
Section
\ref{sec:cpu_cost} discusses computational efficiency of both methods.

\section{Preliminaries}\label{sec:prelim}
Let $X$ be a Hilbert space of functions defined on a domain $D$  and let $\|\cdot \|_X$ denote its norm.
The $L^2$ inner product on a domain $D$ is denoted $(\cdot, \cdot)_D$ or $\int_D \cdot$ , while
$\langle\cdot, \cdot\rangle_D$ denotes the corresponding duality pairing between a Hilbert space $X$ 
and its dual space $X^*$. We will use $H^m= H^m(D)$ to denote the Sobolev space of
functions on $D$ with $m$ derivatives in $L^2 = L^2(D)$. 
The corresponding norm is denoted $\|\cdot\|_{m,D}$. In general, we will use 
$H^m_0$ to denote the closure in $H^m$ of the space of smooth functions with compact
support in $D$ and seminorm is denoted as $|\cdot|_{m, D}$. 

The space of bounded linear operators mapping elements of $X$ to $Y$
is denoted  $\mathcal{L}(X,Y)$ and if $Y= X$ we simply write $\mathcal{L}(X)$ 
instead of $\mathcal{L}(X,X)$. If $X$ and $Y$ are Hilbert spaces, both 
continuously contained in some larger Hilbert space, then the intersection 
$X \cap Y$ and the sum $X+ Y$ are both Hilbert spaces with norms given by
\[
\|x \|_{X \cap Y}^2 = \|x \|_{X}^2 + \|x \|_{Y}^2 \quad \text{and }
\|z \|_{X+Y}^2 = \inf_{\substack{x\in X, y \in Y\\z = x+y}}
(\|x \|_{X}^2 + \|y \|_{Y}^2 ).
\]

In the following $\Omega\subset\reals^n$ is an open connected domain
with Lipschitz boundary $\partial\Omega$.  The trace operator $T$ is
defined by $T u = u|_\Gamma$ for $u\in C(\overline\Omega)$ and
$\Gamma$ a Lipschitz submanifold of codimension one in $\Omega$. The
trace operator extends to bounded and surjective linear operator
$T:H^1(\Omega)\rightarrow H^{\half}(\Gamma)$, see e.g. \cite[ch. 7]{AdamsFournier}. 
The fractional Sobolev space $H^{\half}(\Gamma)$ can be equipped with the norm
\begin{equation}\label{eq:H_half_norm}
  \norm{u}^2_{H^{\half}(\Gamma)} =  \norm{u}^2_{L^2(\Gamma)} +  
  \int_{\Gamma \times \Gamma} \frac{\abs{u(x) -
      u(y)}^2}{\abs{x-y}^{n+1}}\,\mathrm{d}x\mathrm{d}y.
\end{equation}

However, the trace is not surjective as an operator from
$H^1_0(\Omega)$ into $H^{\half}(\Gamma)$, in particular the constant
function $1\in H^{\half}(\Gamma)$ is not in the image of the trace
operator. Note that $H^{\half}_0(\Gamma)$ does not characterize the
trace space, since $H^{\half}_0(\Gamma) = H^{\half}(\Gamma)$, see
\cite[ch. 2, thm. 11.1]{LionsMagenes}. Instead, the trace space
can be identified as $H^{\half}_{00}(\Gamma)$, defined as the subspace of $H^{\half}(\Gamma)$ for
which extension by zero into $H^{\half}(\tilde \Gamma)$ is continuous,
for some suitable extension domain $\tilde \Gamma$ extending $\Gamma$ (e.g. $\tilde
\Gamma = \Gamma \cup \partial\Omega$). To be precise, the space $H^{\half}_{00}(\Gamma)$ 
can be characterized with the norm
\begin{equation}\label{eq:H_half_00_norm}
  \norm{u}_{H_{00}^{\half}(\Gamma)} = \norm{\tilde u}_{H^{\half}(\tilde \Gamma)},
  \quad \tilde u(x) =
  \begin{cases}
    u(x) & x\in \Gamma \\
    0 & x\notin \Gamma.
  \end{cases}
\end{equation}
The space $H_{00}^{\half}(\Gamma)$ does not depend on the extension
domain $\tilde \Gamma$, since the norms induced by different choices of $\tilde \Gamma$
will be equivalent.

The above norms \eqref{eq:H_half_norm}--\eqref{eq:H_half_00_norm} for
the fractional spaces are impractical from an implementation point of
view, and we will therefore consider the alternative construction
following \cite[ch. 2.1]{LionsMagenes} and \cite{CW-H-M}. For $u,
v\in H_0^1(\Gamma)$, set $L_u(v) = (u, v)_{\Gamma}$.  Then $L_u$ is a
bounded linear functional on $H_0^1(\Gamma)$ and in accordance with
the Riesz-Fr\'echet theorem there is an operator $S \in
\mathcal{L}\big(H^1_0(\Gamma)\big)$ such that
\begin{equation}\label{eq:scale_S}
  (Su, w)_{H^1_0(\Omega)} 
  = L_u(w) = 
  \open{u, w}_{\Gamma}, \qquad u,w\in H^1_0(\Gamma).
\end{equation}
The operator $S$ is self-adjoint, positive definite, injective and
compact.  Therefore the spectrum of $S$ consists of a nonincreasing sequence of positive
eigenvalues $\set{\lambda_k}_{k=1}^{\infty}$ such that $0 <
\lambda_{k+1} \leq \lambda_{k}$ and $\lambda_{k}\rightarrow 0$, see
e.g. \cite[ch. X.5, thm. 2]{Yosida}. The eigenvectors $\set{\phi_k}_{k=1}^\infty$
of $S$ satisfy the generalized eigenvalue problem
\begin{equation*}
  A\phi_k = \lambda^{-1}_k M \phi_k
\end{equation*}
where operators $A, M$ are such that $\brack{A u, v}_{\Gamma} =
\open{\nabla u, \nabla v}_{\Gamma}$ and $\brack{M u, v}_{\Gamma} =
\open{u, v}_{\Gamma}$. The set of eigenvectors
$\set{\phi_k}_{k=1}^{\infty}$ forms a basis of $H^1_0\open{\Gamma}$ 
orthogonal with respect the inner product of $H^1_0(\Gamma)$ and orthonormal 
with respect to the inner product on
$L^2(\Gamma)$. Then for $u = \sum_k c_k \phi_k \in \spn\set{\phi_k}_{k=1}^\infty$
and $s\in [-1,1]$, we set
\begin{equation}\label{eq:scale_norm}
  \norm{u}_{H_s} 
  = 
  \sqrt{\sum_k c^2_k \lambda^{-s}_k}
\end{equation}
and define $H_s$ to be the closure of $\spn\set{\phi_k}_{k=1}^\infty$
in the above norm.  Then $H_0 = L^2(\Gamma)$ and $H_1 = H^1_0(\Gamma)$, with
equality of norms. Moreover, we have $H_{\half}=
H^{\half}_{00}(\Gamma)$ with equivalence of norms. This essentially
follows from the fact that $H_{\half}$ and $H^{\half}_{00}(\Gamma)$
are closely related interpolation spaces, see \cite[thm. 3.4]{CW-H-M}. Note that we also  have
$H_{-1}=(H^1_0(\Gamma))^*=H^{-1}(\Gamma)$ and
$H_{-\half}=(H^{\half}_{00}(\Gamma))^*=H^{-\half}(\Gamma)$.

As the preceeding paragraph suggests we shall use normal font to 
denote linear operators, e.g. $A$. To signify that the particular operator acts 
on a vector space with multiple components we employ calligraphic font, e.g. $\mathcal{A}$.  
Vectors and matrices are denoted by the sans serif font, e.g.,  $\mat{A}$ and
$\vec{x}$. In case the matrix has a block structure it is typeset with the 
blackboard bold font, e.g. $\mathbb{A}$. Matrices and vectors are related to the 
discrete problems as follows, see also
\cite[ch. 6]{kent_ragnar}. Let $V_h\subset H^1_0(D)$ and let the 
discrete operator $A_h:V_h\rightarrow V^{*}_h$ be defined in terms of the Galerkin method:
\[
\brack{A_h u_h, v_h}_D = \brack{A u, v_h}_D, \mbox{ for } u_h, v_h \in V_h 
\mbox{ and } u \in H^1_0(D)  .
\]
Let $\psi_j, j\in\close{1, m}$ the basis functions of $V_h$. The matrix equation,
\[
\Amat \vec{u} = \vec{f}, \quad 
\vec{u} \in \mathbb{R}^m  \mbox{ and } \vec{f} \in \mathbb{R}^m
\]
is obtained as follows:  
Let $\pi_h:V_h\rightarrow\reals^m$ and $\mu_h:V^{*}_h\rightarrow \reals^m$ be 
given by
\[
v_h=\sum_j\open{\pi_h v_h}_j\psi_j,\quad v_h\in V_h
\quad\quad\text{and}\quad\quad
\open{\mu_h f_h}_j = \brack{f_h, \psi_j}_D, \quad f_h\in V^{*}_h.
\]
Then 
\[
\Amat = \mu_h A_h \pi_h^{-1}, \quad \vec{v}=\pi_h v_h, \quad \vec{f} = \mu_h f_h.
\]

A discrete equivalent to the $H_s$ inner product \eqref{eq:scale_norm} is 
constructed in the following manner, similar to the continuous case.
There exists a complete set of eigenvectors 
$\vec{u}_i\in\reals^m$ with the property $\transp{\vec{u}_j} \Mmat \vec{u}_i=\delta_{ij}$ 
and $m$ positive definite (not necessarily distinct) eigenvalues $\lambda_i$ of 
the generalized eigenvalue problem $\Amat \vec{u}_i=\lambda_i\Mmat \vec{u}_i$. 
Equivalently the matrix $\Amat$ can be decomposed as 
$\Amat=\left(\Mmat\mat{U}\right)\mat{\Lambda} \transp{\left(\Mmat\mat{U}\right)}$ 
with $\mat{\Lambda}=\text{diag}\left(\lambda_1, \cdots, \lambda_m\right)$ and 
$\text{col}_i \mat{U}=\vec{u}_i$ so that $\transp{\mat{U}}\Mmat\mat{U}=\mat{I}$
and $\transp{\mat{U}}\Amat\mat{U}=\Lambda$.
We remark that $\Amat$ is the stiffness matrix,  while $\Mmat$ is the mass matrix.  
 
Let now $\mat{H}:\reals \rightarrow \mat{P}_{\text{sym}}$, where $\mat{P}_{\text{sym}}$
denotes the space of symmetric positive definite matrices, be defined as
\begin{equation}\label{eq:H_def}
\Hmat{s} = \left(\Mmat\mat{U}\right)
\Lambda^s
\transp{\left(\Mmat\mat{U}\right)}.
\end{equation}
Note that due to $\mat{M}$ orthonormality of the eigenvectors the inverse of
$\Hmat{s}$ is given as $\inv{\Hmat{s}}=\mat{U}\mat{\Lambda}^{-s}\transp{\mat{U}}$. 
To motivate the definition of the mapping, we shall in the following example
consider several values $\Hmat{s}$ and show the relation of the matrices to
different Sobolev (semi) norms of functions in $V_h$.

\begin{example}[$L_2$, $H^1_0$ and $H^{-1}$ norms in terms of matrices]\label{ex:norms}
Let $V_h\subset H^1_0\left(\Gamma\right)$, $\dim V_h=m$,  $v_h\in V_h$ and 
$\vec{v}\in\reals^m$ the representation of $v_h$ in the basis of $V_h$, i.e. $\vec{v} = \pi_h v_h$. The $L^2$
norm of $v_h$ is given through the mass matrix 
$\mat{M}$ as $\|v_h\|_{0, \Gamma}^2=\transp{\vec{v}}\Mmat\vec{v}$ 
and $\Mmat=\Hmat{0}$. Similarly for the $H^1_0$ (semi) norm it holds
that $\semi{v_h}^2_{1, \Gamma}=\transp{\vec{v}}\Amat\vec{v}$, where $\Amat$ is 
the stiffness matrix, and  $\Amat=\Hmat{1}$. 
Finally a less trivial example, let $f_h \in V_h$ and consider $f_h$ as a
bounded linear functional, $\brack{f_h, v_h}_{\Gamma}=\inner{f_h}{v_h}_{\Gamma}$ 
for $v_h\in V_h$. Then $\norm{f_h}^2_{-1, \Gamma}=\transp{\vec{f}}\Hmat{-1}\vec{f}$. 
By Riesz representation theorem there exists a unique $u_h \in V_h$ such that
$\inner{\nabla u_h}{\nabla v_h}_{\Gamma} = \brack{f_h, v_h}_{\Gamma}$ for all 
$v_h\in V_h$ and $\|f_h\|_{-1, \Gamma}=\semi{u_h}_{1, \Gamma}$. The latter equality 
yields $\|f_h\|^2_{-1, \Gamma}=\transp{\vec{u}}\Amat\vec{u}$ but since $u_h\in V_h$ is 
given by the Riesz map, the coordinate vector comes as a unique solution of the system
$\Amat\vec{u} = \Mmat \vec{f}$, i.e. $\vec{u}=\inv{\Amat}\Mmat\vec{f}$ 
(see e.g. \cite[ch. 3]{malek}). Thus 
$\|f_h\|^2_{-1, \Gamma}=\transp{\vec{f}}\Mmat\inv{\Amat}\Mmat\vec{f}$. The matrix
product in the expression is then $\Hmat{-1}$.
\end{example}

In general let $\vec{c}$ be the representation of vector $\vec{u}\in\reals^m$ in
the basis of eigenvectors $\vec{u}_i$, $\vec{u}=\mat{U}\vec{c}$. Then 
\[
\transp{\vec{u}}\Hmat{s}\vec{u} = \transp{\vec{c}}\mat{\Lambda}^{s}\vec{c} = 
\sum_j c^2_j \lambda^{s}_j
\]
and so $\transp{\vec{u}}\Hmat{s}\vec{u}=\norm{u_h}^2_{{H}_{s}}$
for $u_h\in V_h$ such that  $u_h=\pi^{-1}_h\vec{u}$. Similar to the continuous 
case the norm can be obtained in terms of powers of an operator
\[
\transp{\vec{u}}\Hmat{s}\vec{u}
=
\transp{\close{\mat{U} \Lambda^{\tfrac{s}{2}} \transp{\open{\mat{MU}}} \vec{u}}}
\mat{M}
\close{\mat{U} \Lambda^{\tfrac{s}{2}} \transp{\open{\mat{MU}}} \vec{u}}
=
\transp{\close{\mat{S}^{-\tfrac{s}{2}} \vec{u}}}
\mat{M}
\close{\mat{S}^{-\tfrac{s}{2}} \vec{u}},
\]
where $\mat{S}=\inv{\mat{A}}\mat{M}$ is the matrix representation of the Riesz map
$H^{-1}\open{\Gamma}\rightarrow H^1_0\open{\Gamma}$ in the basis of $V_h$.
\begin{remark}
  The norms constructed above for the discrete space are equivalent
  to, but not identical to the $H_s$-norm from the continuous case.
\end{remark}

Before considering proper preconditioning of the weak formulation of
problem \eqref{pde} we illustrate the use of operator preconditioning with an
example of a boundary value problem where operators in fractional spaces are
utilized to weakly enforce the Dirichlet boundary conditions by Lagrange 
multipliers~\cite{babuska_lm}.

\begin{example}[Dirichlet boundary conditions using Lagrange multiplier]  
\label{ex:babuska_lm} 
The problem considered in \cite{babuska_lm} reads: Find $u$ such that 
\begin{equation}\label{eq:babuska_lm}
\begin{aligned}
  -\Delta u + u &= f &\text{ in }\Omega,\\
              u &= g   &\text{ on }\Gamma\subset\partial\Omega,\\
  \partial_{n}u &= 0 &\text{ on }\partial\Omega\setminus\Gamma.
\end{aligned}
\end{equation}
Introducing a Lagrange multiplier $p$ for the boundary value constraint and a 
trace operator $T: H^1(\Omega) \rightarrow H^{\half}(\Gamma)$ leads
to a variational problem for $\left(u, p\right) \in H^1\left(\Omega\right)\times
H^{-\half}\left(\Gamma\right)$ satisfying
\begin{equation}\label{eq:babuska}
\begin{aligned}
  &\inner{\nabla u}{\nabla v}_{\Omega} + \inner{u}{v}_{\Omega} + \brack{p,
  T v}_\Gamma = \inner{f}{v}_{\Omega} & v\in H^1\left(\Omega\right),\\
  &\brack{q, T u}_\Gamma = \brack{q, g}_{\Gamma} & q\in H^{-\half}\left(\Gamma\right).
  \end{aligned}
\end{equation}
In terms of the framework of operator preconditioning, 
the variational problem \eqref{eq:babuska} defines an equation
\begin{equation}\label{eq:Axb0}
  \mathcal{A}x=b, \quad \mbox{where}  
\quad 
\mathcal{A} = \begin{bmatrix}
  -\Delta_{\Omega} + I & T' \\
    T  & 0  \\
\end{bmatrix}.
\end{equation}
In \cite{babuska_lm} the problem is proved to be well-posed and
therefore  
$\mathcal{A}:V\rightarrow V^*$ is a symmetric isomorphism, where
$V=H^1\open{\Omega}\times H^{-\half}\open{\Gamma}$ and $x\in V$, $b\in V^*$. 
A preconditioner is then
$\mathcal{B}\in\mathcal{L}\open{V^*, V}$, constructed such that $\mathcal{B}$ is 
a positive, self-adjoint isomorphism. Then 
$\mathcal{BA}\in\mathcal{L}\open{V}$ is an isomorphism. 

To discretize \eqref{eq:Axb0} we shall here employ finite element spaces $V_h$ 
consisting of linear continuous finite elements where $\Gamma_h$
  is formed by the facets of $\Omega_h $, cf. Figure \ref{fig:scheme}.
  Stability of discretizations of \eqref{eq:babuska} (for the more general case where the discretization of $\Omega$ and $\Gamma$ are independent) 
   is studied e.g. in \cite{pitkaranta_lm} and \cite[ch. 11.3]{steinbach_book}.
  
The linear system resulting from discretization leads to the following 
system of equations
\begin{equation}
\label{eq:BAxBbh}
\mathbb{B} \mathbb{A} \vec{x} = \mathbb{B} \vec{b},  
\end{equation}
where
\[
\mathbb{B} = 
\begin{bmatrix}
                \inv{\Amat}  &  \\
                   & \inv{\Hmat{-\tfrac{1}{2}}}\\
                \end{bmatrix}
\quad \mbox{and} \quad 
\mathbb{A} = 
\begin{bmatrix}
                \Amat  &  \transp{\mat{B}}\\
                \mat{B} &      \\
                \end{bmatrix} . 
\]
The last block of the matrix preconditioner $\mathbb{B}$ is the inverse of the 
matrix constructed by \eqref{eq:H_def} (using discretization of an operator inducing 
the $H^1(\Gamma)$ norm on the second subspace of $V_h$) and matrix $\mathbb{B}\mathbb{A}$ 
has the same eigenvalues as operator $\mathcal{B}_h \mathcal{A}_h$.   

Tables \ref{tab:babuska_eigs} and \ref{tab:babuska_iters} consider the problem
\eqref{eq:babuska} with $\Omega$ the unit square and $\Gamma$ its left edge. In 
Table \ref{tab:babuska_eigs} we show the spectral condition number of the matrix 
$\mathbb{B} \mathbb{A}$ as a function of the discretization parameter $h$. 
It is evident that the condition number is bounded by a constant.

Table \ref{tab:babuska_iters} then reports the number of iterations required for 
convergence of the minimal residual method \cite{minres} with the 
system \eqref{eq:BAxBbh} of different sizes. The iterations are started from a random 
initial vector and for convergence it is required that $\vec{r}_k$, the $k$-th 
residuum, satisfies 
$\transp{\vec{r}_k}\overbar{\mathbb{B}} \vec{r}_k <
10^{-10}$. The operator $\overbar{\mathbb{B}}$ is the spectrally equivalent
approximation of $\mathbb{B}$ given as\footnotemark
\begin{equation}\label{eq:Bh_approx}
\overbar{\mathbb{B}} = \diag\open{\mbox{\normalfont{AMG}}\open{\mat{A}},
                            \mbox{\normalfont{LU}}\open{\mat{H}\open{-\tfrac{1}{2}}}}.
\end{equation}
The iteration count appears to be bounded independently of the size of the linear system. 
\footnotetext{
Here and in the subsequent numerical experiments AMG is the algebraic multigrid 
BOOMERAMG from the Hypre library \cite{Hypre} and LU is the direct solver from the UMFPACK 
library \cite{umfpack}. The libraries were accessed through the interaface
provided by PETSc \cite{petsc} version 3.5.3. To assemble the relevant matrices 
FEniCS library \cite{fenics} version 1.6.0 and its extension for block-structured 
systems cbc.block \cite{block} were used. The AMG preconditioner was used with
the default options except for coarsening which was set to Ruge-Stueben algorithm.
}

\begin{table}[t!]
  \centering
  \setlength\tabcolsep{4pt}
  \begin{minipage}{0.40\textwidth}
    \centering
    \caption{The smallest and the largest eigenvalues and the spectral condition 
    number of matrix $\BBh \AAh$ from system \eqref{eq:BAxBbh}.
    }
    \footnotesize{
    \begin{tabular}{c|ccc}
      \hline
      $h$ & $\lambda_{\text{min}}$ & $\lambda_{\text{max}}$ & $\kappa$ \\
      \hline
      $1.77 \times 10^{-1}$ & 0.311 & 1.750 & 5.622\\
      $8.84 \times 10^{-2}$ & 0.311 & 1.750 & 5.622\\
      $4.42 \times 10^{-2}$ & 0.311 & 1.750 & 5.622\\
      $2.21 \times 10^{-2}$ & 0.311 & 1.750 & 5.622\\
      $1.11 \times 10^{-2}$ & 0.311 & 1.750 & 5.622\\
      \hline
    \end{tabular}
    }
    \label{tab:babuska_eigs} 
  \end{minipage}%
\hfill
  \begin{minipage}{0.55\textwidth}
  \centering
     \caption{The number of iterations required for convergence of the minimal
     residual method for system \eqref{eq:BAxBbh} with $\BBh$ replaced
     by the approximation \eqref{eq:Bh_approx}.
     } 
  \footnotesize{
    \begin{tabular}{c|c c}
    \hline
    size & $n_{\text{iters}}$ & $\norm{u-u_h}_{1, \Omega}$ \\
    \hline
4290 & 38 & $6.76 \times 10^{-2}$(1.00)\\
16770 & 40 & $3.38 \times 10^{-2}$(1.00)\\
66306 & 38 & $1.69 \times 10^{-2}$(1.00)\\
263682 & 38 & $8.45 \times 10^{-3}$(1.00)\\
1051650 & 39 & $4.23 \times 10^{-3}$(1.00)\\
    \hline
    \end{tabular}
     }
    \label{tab:babuska_iters} 
  \end{minipage}
\end{table}

Together the presented results indicate that the constructed preconditioner whose
discrete approximation utilizes matrices \eqref{eq:H_def} is a
good preconditioner for system \eqref{eq:babuska_lm}. 
\end{example}

Finally, with $\Omega\in\reals^2$, $\Gamma\subset\Omega$ of codimension one we
consider the problem \eqref{pde}. The weak formulation of \eqref{pde1}--\eqref{coupling}, 
using the method of Lagrange multipliers, defines a variational problem for the triplet 
$\left(u, v, p\right)\in U \times V \times Q$
\begin{equation}
  \label{eq:weak_form}
\begin{aligned}
  \inner{\nabla u}{\nabla \phi}_{\Omega} + \brack{p, \epsilon T_\Gamma \phi}_\Gamma &=
  \inner{f}{\phi}_{\Omega} 
  \quad\quad&\phi &\in U,
  \\
  \inner{\nabla v}{\nabla \psi}_{\Gamma} - \brack{p, \psi}_\Gamma &=
  \inner{g}{\psi}_{\Gamma}
  \quad\quad&\psi &\in V,
  \\
  \brack{\chi, \epsilon T_\Gamma u - v}_\Gamma &= 0
  \quad\quad&\chi &\in Q,
\end{aligned}
\end{equation}
where $U, V, Q$ are Hilbert spaces to be specified later. 
The well-posedness of \eqref{eq:weak_form} is guaranteed provided that the 
celebrated Brezzi conditions, see Appendix A, are fulfilled. 
We remark that 
\[
\brack{p, T_\Gamma \phi}_\Gamma = \brack{\delta_\Gamma p,  \phi}_\Omega.  
\]
Hence $\delta_\Gamma$ is in our context the dual operator to the trace operator $T_\Gamma$.
Since $T_\Gamma: H^1_0(\Omega) \rightarrow H^{\half}_{00}(\Gamma)$,
then $\delta_\Gamma:  H^{\nhalf}(\Gamma) \rightarrow H^{-1}(\Omega)$.
 
For our discussion
of preconditioners it is suitable to recast \eqref{eq:weak_form} as an operator
equation for the self-adjoint operator $\mathcal{A}$
\begin{equation}
  \label{eq:op_long}
\mathcal{A}\begin{bmatrix}u \\v \\p \end{bmatrix}=
\begin{bmatrix}
  \AU &     & \BU^{*}\\
      & \AV & \BV^{*}\\
  \BU & \BV &       \\
\end{bmatrix}\begin{bmatrix}u \\ v \\ p\end{bmatrix}=
\begin{bmatrix}f \\ g \\ \mbox{ } \end{bmatrix}
\end{equation}
with the operators $A_i, B_i, i\in\set{U, V}$ given by
\[
\begin{aligned}
\brack{\AU u, \phi}_\Omega &= \inner{\nabla u }{\nabla \phi}_{\Omega},
&\brack{\AV v, \psi}_\Gamma &= \inner{\nabla v}{\nabla \psi}_{\Gamma},
\\
\brack{\BU u, \chi}_\Gamma   &= \brack{\chi, \epsilon T_\Gamma u}_\Gamma,
&\brack{\BV v, \chi}_\Gamma &=-\brack{\chi, v}_\Gamma.
\end{aligned}
\]
Further, for discussion of mapping properties of $\mathcal{A}$ it will be 
advantageous to consider the operator as a map defined over space $W\times Q$, 
$W=U\times V$ as
\begin{equation}\label{eq:op_short}
\mathcal{A}=
\begin{bmatrix}
A & B^*\\
B & \\
\end{bmatrix}
\quad \mbox{with}
\quad 
A= \begin{bmatrix}
A_U & \\
 & A_V \\
\end{bmatrix}
\quad \mbox{and}
\quad 
B= \begin{bmatrix}
B_U & B_V 
\end{bmatrix}.
\end{equation}

Considering two different choices of spaces $U, V$ and $Q$ we will propose 
two formulations that lead to different preconditioners
\begin{equation}\label{eq:operator_preconditionersQ}
  \BBQ^{-1}
  =
  \begin{bmatrix}
    \AU &     & \\
        & \AV & \\
        &   & \BU \AU^{-1} \BU^{*}  + \BV \AV^{-1} \BV^{*}      \\ 
  \end{bmatrix}
\end{equation}
and
\begin{equation}\label{eq:operator_preconditionersV}
\BBW^{-1}
  =
  \begin{bmatrix}
    \AU + \BU ^{*} R \BU &     & \\
        & \AV & \\
        &  &   \BV \AV^{-1} \BV^{*}       \\
  \end{bmatrix}.
\end{equation}
Here $R$ is the Riesz map from $Q^*$ to $Q$. Preconditioners of the form
\eqref{eq:operator_preconditionersQ}--\eqref{eq:operator_preconditionersV} will
be referred to as the $Q$-cap and the $W$-cap preconditioners. This naming 
convention reflects the role intersection spaces play in the respected 
formulations. We remark that the definitions should be understood as templates 
identifying the correct structure of the preconditioner.

\section{$Q$-cap preconditioner}\label{sec:Qcap}
Consider operator $\mathcal{A}$
from problem \eqref{eq:op_long} as a mapping $W\times Q \rightarrow W^*\times Q^*$, 
\begin{equation}
  \label{eq:V_Q_1}
  \begin{aligned}
    W  &= H^1_0\left(\Omega\right)\times  H^1_0\left(\Gamma\right),
    \\
    Q  &=\epsilon H^{-\half}(\Gamma) \cap H^{-1}\left(\Gamma\right).
  \end{aligned}
\end{equation}
The spaces are equipped with norms
\begin{equation}\label{eq:eigen_norms}
\norm{w}^2_W = \HnormO{u}^2 + \HnormG{v}^2 \quad
  \text{ and }\quad
  \norm{p}^2_Q = \epsilon^2\norm{p}_{-\half, \Gamma}^2 + \norm{p}^2_{-1, \Gamma}.
\end{equation}
Since $H^{-\half}(\Gamma)$ is continuously embedded in
$H^{-1}(\Gamma)$, the space $Q$ is
the same topological vector space as $H^{-\half}(\Gamma)$, but
equipped with an equivalent, $\epsilon$-dependent inner product.
See also \cite[ch. 2]{BerghLofstrom}. The next theorem shows that this 
definition leads to a well-posed problem.

We will need a right inverse of the trace operator and employ the following
harmonic extension. Let $q \in H^{\half}_{00}(\Gamma)$  and
let $u$ be the solution of the problem
\begin{equation}\label{eq:harmonic_ext}
\begin{aligned}
-\Delta u &= 0,  &\mbox{ in }  \Omega \setminus \Gamma, \\  
       u &= 0,   &\mbox{ on }  \partial \Omega,  \\ 
       u &= q,   &\mbox{ on }  \Gamma.   
\end{aligned}
\end{equation}
Since trace is surjective onto $H^{\half}_{00}(\Gamma)$,
\eqref{eq:harmonic_ext} has a solution $u\in H^1_0(\Omega)$ and
$|u|_{1, \Omega} \le C |q|_{{\half}, \Gamma}$ for some constant
$C$. We denote the harmonic extension operator by $E$, i.e., $u = E
q$ with $\norm{E}  \leq C$.
\begin{theorem}\label{thm:stab_eig}
Let $W$ and $Q$ be the spaces \eqref{eq:V_Q_1}. The operator $\mathcal{A}:W\times Q\rightarrow W^*\times Q^*$,
defined in \eqref{eq:op_long} is an isomorphism and the condition number of 
$\mathcal{A}$ is bounded independently of $\epsilon >0$.
\end{theorem}
\begin{proof}
  The statement follows from the Brezzi theorem \ref{thm:brezzi} once its 
  assumptions are verified.
Since $A$ induces the inner product on $W$, $A$ is continuous and coercive and
the conditions \eqref{eq:Brezzi_A_bounded} and \eqref{eq:Brezzi_A_coercivity} hold.
Next, we see that $B$ is bounded,
\begin{equation*}
  \begin{split}
    \brack{Bw, q}_{\Gamma}
    & = \brack{q, \epsilon T_\Gamma u - v}_{\Gamma}
    \\
    &\leq
    \norm{q}_{{\nhalf}, \Gamma}\norm{\epsilon T_\Gamma u}_{{\half}, \Gamma}
    + \norm{q}_{{-1}, \Gamma}\semi{v}_{1, \Gamma}
    \\
    &\leq \big(1+\norm{T_\Gamma}\big)
    \sqrt{\epsilon^2\norm{q}_{\nhalf, \Gamma}^2 + \norm{q}_{{-1}, \Gamma}^2}
    \sqrt{\semi{u}_{1, \Omega}^2 + \semi{v}_{{1}, \Gamma}^2}
    \\
    & = \big(1+\norm{T_\Gamma}\big)
    \norm{q}_Q\norm{w}_W.
  \end{split}
\end{equation*}
It remains to show the inf-sup
condition \eqref{eq:Brezzi_infsup}. Since the trace is bounded and surjective, for all
$\xi \in H^{\half}_{00}(\Gamma)$ we let $u$ be defined in terms of the harmonic
extension \eqref{eq:harmonic_ext}  such that $ u = \epsilon^{-1} E \xi$
and $\semi{u}_{1,\Omega}\leq \epsilon^{-1}\norm{E}\norm{\xi}_{\half, \Gamma}$. Hence,
\begin{equation*}
\begin{split}
  \sup_{w\in W}\, \frac{\brack{B w, q}_{\Gamma}}{\norm{w}_{W}}
  &= \sup_{w\in W} 
  \frac{\brack{q, \epsilon T_\Gamma u - v}_{\Gamma}} {\sqrt{\HnormO{u}^2 + \HnormG{v}^2}}
  \\
  &\geq \big(1 + \norm{E}\big)^{-1}
  \sup_{(\xi, v)\in H^{\half}_{00}(\Gamma) \times \HG}\, \frac{\brack{q, \xi + v}_{\Gamma}}
  {\sqrt{\epsilon^{-2}\norm{\xi}_{\half, \Gamma}^2 + \norm{v}_{1, \Gamma}^2}}
\end{split}
\end{equation*}
Note that we have the identity
\begin{equation*}
  Q^* = \big(\epsilon {H}^{-\half}(\Gamma) \cap 
      H^{-1}(\Gamma) \big)^*
      = \epsilon^{-1} {H}^{\half}_{00}(\Gamma) +
      H^{1}_0(\Gamma),
\end{equation*}
equipped with the norm
\begin{equation*}
  \norm{q^*}_{Q^*} = 
  \inf_{q^* = q^*_1+q^*_2} 
  \epsilon^{-2}\norm{q_1^*}_{\half, \Gamma}^2
  + \semi{q_2^*}_{1, \Gamma}^2.
\end{equation*}
See also \cite{BerghLofstrom}. It follows that
\begin{equation*}
    \begin{split}
    \sup_{(\xi, v)\in H^{\half}(\Gamma) \times \HG}\, \frac{\brack{q, \xi + v}_{\Gamma}}
    {\sqrt{\epsilon^{-2}\norm{\xi}_{\half, \Gamma}^2 + \HnormG{v}^2}}
    &=\sup_{\zeta \in Q^*}\, \sup_{\substack{
        \xi + v = \zeta \\  
        v \in H^1_0(\Gamma)
        }}\,
    \frac{\brack{q, \xi + v}_{\Gamma}}
    {\sqrt{\epsilon^{-2}\norm{\xi}_{\half, \Gamma}^2 + \HnormG{v}^2}}
    \\
    &=\sup_{\zeta \in Q^*}
    \frac{\brack{q,\zeta}_{\Gamma}}
    {\displaystyle\inf_{\substack{
        \xi + v = \zeta \\  
        v \in H^1_0(\Gamma)
        }}
      \sqrt{\epsilon^{-2}\norm{\xi}_{\half, \Gamma}^2 + \HnormG{v}^2}}
    \\
    &=\norm{q}_{Q^{**}}
    =\norm{q}_{Q}.
  \end{split}
\end{equation*}
Consequently, condition \eqref{eq:Brezzi_infsup} holds 
with a constant independent of $\epsilon$. \qquad
\end{proof}

Following Theorem \ref{thm:stab_eig} and \cite{kent_ragnar} a preconditioner for
the symmetric isomorphic operator $\mathcal{A}$ is the Riesz mapping $W^*\times Q^*$
to $W\times Q$
\begin{equation}\label{eq:Qcap_op}
\BBQ=
\begin{bmatrix}
  -\Delta_{\Omega} & & \\
                   & -\Delta_{\Gamma} & \\
		   & &   & \epsilon^2\Delta_{\Gamma}^{\nhalf} +
        \Delta_{\Gamma}^{-1}
\end{bmatrix}^{-1}.
\end{equation}
Here $\Delta_\Gamma^s$ is defined by $\brack{\Delta_\Gamma^s v, w}_\Gamma =
(v, w)_{H_s}$, with the $H_s$-inner product defined by
\eqref{eq:scale_norm}.  Hence the norm induced on
$W\times Q$ by the operator $\BBQ^{-1}$ is not \eqref{eq:eigen_norms}
but an equivalent norm
\[
  \brack{\inv{\BBQ} x, x}
  = \semi{u}_{1, \Omega}^2 + \semi{v}_{1, \Gamma}^2
  + \epsilon^2\norm{p}^2_{H_{\nhalf}(\Gamma)}
  + \norm{p}^2_{H_{-1}(\Gamma)}
\]
for any $x=\open{u, v, p}\in W\times Q$.  Note that $\BBQ$ fits the template defined in
$\eqref{eq:operator_preconditionersQ}$.
\subsection{Discrete $Q$-cap preconditioner}\label{sec:Qcap_discrete}
\begin{figure}[ht!]
\centering
%
%
%
\begin{subfigure}[b]{0.35\textwidth}
    \includegraphics[height=0.2\textheight]{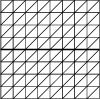} 
  \caption{}
	  \label{fig:mesh_1}
\end{subfigure}
\hfill
\begin{subfigure}[b]{0.35\textwidth}
    \includegraphics[height=0.201\textheight]{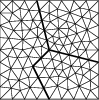} 
  \caption{}
  \label{fig:mesh_4}
\end{subfigure}
\caption{
Geometrical configurations and their sample triangulations considered in the numerical 
experiments.} 
\label{fig:scheme}
\end{figure}

Following Theorem \ref{thm:stab_eig} the $Q$-cap preconditioner \eqref{eq:Qcap_op}
is a good preconditioner for operator equation $\mathcal{A}x=b$ with the
condition number independent of the material parameter $\epsilon$. To translate
the preconditioned operator equation $\BBQ\mathcal{A}x=\BBQ b$ into
a stable linear system it is necessary to employ suitable discretization. In
particular, the Brezzi conditions must hold on each approximation space 
$W_h\times Q_h$ with constants independent of the discretization parameter $h$.
Such a suitable discretization will be referred to as stable.

Let us consider a stable discretization of operator $\mathcal{A}$ from Theorem
\ref{thm:stab_eig} by finite dimensional spaces $U_h, V_h$ and $Q_h$ defined as
\[
U_h = \spn\set{{\phi_i}}_{i=1}^{n_{{U}}},
\quad
V_h = \spn\set{{\psi_i}}_{i=1}^{n_{{V}}},
\quad
Q_h = \spn\set{{\chi_i}}_{i=1}^{n_{{Q}}}.
\]
Then the Galerkin method for problem \eqref{eq:op_long} reads: Find $(u_h, v_h,
p_h)\in U_h\times V_h\times Q_h$ such that
\[
\begin{aligned}
  \inner{\nabla u_h}{\nabla \phi}_{\Omega} + \brack{p_h, \epsilon T_\Gamma \phi}_\Gamma &=
  \inner{f}{\phi}_{\Omega} 
  \quad\quad&\phi &\in U_h,
  \\
  \inner{\nabla v_h}{\nabla \psi}_{\Gamma} - \brack{p_h, \psi}_\Gamma &=
  \inner{g}{\psi}_{\Gamma}
  \quad\quad&\psi &\in V_h,
  \\
  \brack{\chi, \epsilon T_\Gamma u_h - v_h}_\Gamma &= 0
  \quad\quad&\chi &\in Q_h.
\end{aligned}
\]
Further we shall define matrices $\mat{A}_{U}$, $\mat{A}_{V}$ and 
$\mat{B}_{U}$, $\mat{B}_{V}$ in the following way
\begin{equation}\label{eq:mat_defs}
\begin{aligned}
\mat{A}_{{U}} &\in\reals^{n_U\times n_U}, &
\open{\mat{A}_{{U}}}_{i, j}&=\inner{\nabla \phi_j}{\nabla \phi_i}_{\Omega},
\\
\mat{A}_{{V}} &\in\reals^{n_V\times n_V}, &
\open{\mat{A}_{{V}}}_{i, j}&=\inner{\nabla \psi_j}{\nabla \psi_i}_{\Gamma},
\\
\mat{B}_{{U}}&\in\reals^{n_Q\times n_U}, &
\open{\mat{B}_{{U}}}_{i, j}&=\brack{\epsilon T_{\Gamma} \phi_j, \chi_i}_{\Gamma},
\\
\mat{B}_{{V}}&\in\reals^{n_Q\times n_V},&
\open{\mat{B}_{{V}}}_{i, j}&=-\brack{\psi_j, \chi_i}_{\Gamma}.
\end{aligned}
\end{equation}
We note that $\mat{B}_{{V}}$ can be viewed as a representation of the negative
identity mapping between spaces $V_h$ and $Q_h$. Similarly, matrix $\mat{B}_{{U}}$ can be 
viewed as a composite, $\mat{B}_{{U}}=\mat{M}_{{\overbar{U}Q}}\mat{T}$. Here
$\mat{M}_{{\overbar{U}Q}}$ is the representation of an identity map from space
$\overbar{U}_h$ to space $Q_h$. The space $\overbar{U}_h$ is the image of $U_h$ under the 
trace mapping $T_{\Gamma}$. We shall respectively denote the dimension of the space and its 
basis functions $n_{{\overbar{U}}}$ and $\overbar{\phi}_i$, $i\in\close{1,
n_{{\overbar{U}}}}$. Matrix $\mat{T}\in\reals^{n_{{\overbar{U}}} \times n_{{U}}}$
is then a representation of the trace mapping $T_{\Gamma}: U_h
\rightarrow\overbar{U}_h$.

We note that the rank of $\mat{T}$ is $n_{{Q}}$ and mirroring the continuous 
operator $T_{\Gamma}$ the matrix has a unique right inverse $\mat{T}^{+}$. We refer 
to \cite{trace_inverse} for the continuous case. The matrix $\mat{T}^{+}$ can 
be computed as a pseudoinverse via the reduced singular value decomposition 
$\mat{T}\mat{U}=\mat{Q}\mat{\Sigma}$, see e.g. \cite[ch. 11]{trefethen}. Then
$\mat{T}^{+}=\mat{U}\inv{\mat{\Sigma}}\mat{Q}$. Here, the 
columns of $\mat{U}$ can be viewed as coordinates of functions $\overbar{\phi}_i$ 
zero-extended to $\Omega$ such that they form the $l^2$ orthonormal basis of the 
subspace of $\reals^{n_{{U}}}$ where the problem $\mat{T}\vec{u}=\overbar{\vec{u}}$ 
is solvable. Further the kernel of $\mat{T}$ is spanned by $n_{U}$-vectors 
representing those functions in $U_h$ whose trace on $\Gamma$ is zero. 

For the space $U_h$ constructed by the finite element method with the
triangulation of $\Omega$ such that $\Gamma$ is aligned with the element
boundaries, cf. Figure \ref{fig:scheme}, it is a consequence of the
nodality of the basis that $\mat{T}^{+}=\transp{\mat{T}}$.

With definitions \eqref{eq:mat_defs} we use $\mathbb{A}$ to represent the operator 
$\mathcal{A}$ from \eqref{eq:op_long} in the basis of $W_h\times Q_h$
\begin{equation}\label{eq:eigen_matrix}
\AAh =
\begin{bmatrix}
\Amat_{{U}} &                 & \transp{\Bmat_{{U}}}\\
                & \Amat_{{V}} & \transp{\Bmat_{{V}}}\\
\Bmat_{{U}} & \Bmat_{{V}} &                         \\
\end{bmatrix}.
\end{equation}
Finally a discrete $Q$-cap preconditioner is defined as a matrix representation
of \eqref{eq:Qcap_op} with respect to the basis of $W_h \times Q_h$
\begin{equation}\label{eq:Qcap_matrix}
\BBQh =
\begin{bmatrix}
\Amat_{{U}} &                 & \\
                & \Amat_{{V}} & \\
                &                 & \epsilon^2 \Hmat{-\tfrac{1}{2}} + \Hmat{-1}\\
\end{bmatrix}^{-1}.
\end{equation}
The matrices $\mat{A}$, $\mat{M}$ which are used to compute the values
$\Hmat{\cdot}$ through the definition \eqref{eq:H_def} have the property 
$\semi{p}^2_{1, \Gamma}=\transp{\vec{p}}\mat{A}\vec{p}$
and $\norm{p}^2_{0, \Gamma}=\transp{\vec{p}}\mat{M}\vec{p}$ for every $p\in Q_h$
and $\vec{p}\in\reals^{n_{{Q}}}$ its coordinate vector. Note that due to
properties of matrices $\Hmat{\cdot}$, matrix $\mat{N}_Q$, the inverse of the 
final block of $\BBQh$, is given by
\begin{equation}\label{eq:Qcap_inverse_formula}
\mat{N}_Q
=
\close{\epsilon^2 \Hmat{-\tfrac{1}{2}} + \Hmat{-1}}^{-1}
= 
\mat{U}\left[\epsilon^{2}\mat{\Lambda}^{\nhalf} +
\mat{\Lambda}^{\Scale[0.5]{-1}}\right]^{-1}\transp{\mat{U}}.
\end{equation}

By Theorem \ref{thm:stab_eig} and the assumption on spaces $W_h\times Q_h$ being
stable, the matrix $\BBQh\AAh$ has a spectrum bounded independent of the parameter $\epsilon$
and the size of the system or equivalently discretization parameter $h$. In turn 
$\BBQh$ is a good preconditioner for matrix $\AAh$. To demonstrate this property 
we shall now construct a stable discretization of the space $W\times Q$ using the 
finite element method. 

\subsection{Stable subspaces for $Q$-cap preconditioner}\label{sec:Qcap_Vh}
For $h>0$ fixed let $\Omega_h$ be the polygonal approximation of $\Omega$. For 
the set $\bar{\Omega}_h$ we construct a shape-regular triangulation consisting of
closed triangles $K_i$ such that $\Gamma\cap K_i$ is an edge $e_i$ of the
triangle. Let $\Gamma_h$ be a union of such edges. The discrete spaces 
$W_h\subset W$ and $Q_h\subset Q$ shall be defined in the following way. 
Let
\begin{equation}\label{eq:hcomps}
\begin{aligned}
U_h &= \set{v\in C\open{\overbar{\Omega}_h}\,:\, v|_{K}=\mathbb{P}_1\open{K}},\\
V_h &= \set{v\in C\open{\overbar{\Gamma}_h}\,:\, v|_{e}=\mathbb{P}_1\open{e}},\\
\end{aligned}
\end{equation}
where $\mathbb{P}_1\open{D}$ are linear polynomials on the simplex $D$. Then we
set 
\begin{equation}\label{eq:hspaces}
\begin{aligned}
W_h &= \open{U_h \cap H^1_0\open{\Omega}} \times \open{V_h \cap H^1_0\open{\Gamma}}, \\
Q_h &= {V_h \cap H^1_0\open{\Gamma}}.
\end{aligned}
\end{equation}

Let $A_h, B_h$ be the finite dimensional operators defined on the approximation spaces \eqref{eq:hspaces}
in terms of Galerkin method for operators $A, B$ in \eqref{eq:op_short}.
Since the constructed spaces are conforming the operators $A_h$, $B_h$ are continuous 
with respect to the norms \eqref{eq:eigen_norms}. Further $A_h$ is $W$-elliptic on 
$W_h$ since the operator defines an inner product on the discrete space. Thus to 
show that the spaces $W_h\times Q_h$ are stable it remains to show that the 
discrete inf-sup condition holds.

\begin{lemma}\label{Qcap_discrete_inf_sup} 
  Let $W_h\subset W$, $Q_h\subset Q$ be the spaces
  \eqref{eq:hspaces}. Further let
$\norm{\cdot}_W$, $\norm{\cdot}_Q$ be the norms \eqref{eq:eigen_norms}. Finally
let $B_h$ such that $\brack{B_h w_h, q_h}_{\Gamma}=\brack{B w, q_h}_{\Gamma}$, $w\in W$. There exists 
a constant $\beta>0$ such that
\begin{equation}\label{eq:discrete_inf_sup}
    \inf_{q_h\in Q_h}\,\sup_{w_h\in W_h}\, 
    \frac{\brack{B_h w_h, q_h}_{\Gamma}}{\norm{w_h}_W\norm{q_h}_Q} \geq \beta.
\end{equation}
\end{lemma}
\begin{proof}
Recall $Q=\epsilon H^{-\half}(\Gamma) \cap H^{-1}\left(\Gamma\right)$.
We follow the steps of the continuous inf-sup condition in the reverse order. By
definition
\begin{equation}\label{eq:1234}
\begin{split}
\norm{q_h}_Q &= \sup_{p\in \epsilon H^{\half}_{00}\open{\Gamma} + H^1_0\open{\Gamma}}
\frac{\brack{q_h, p}_{\Gamma}}{\displaystyle\inf_{p=p_1+p_2}\sqrt{\epsilon^{-2}\norm{p_1}^2_{\half,
\Gamma} +
\semi{p_2}^2_{1, \Gamma}}}\\
&=\sup_{p}\sup_{p=p_1+p_2}
\frac{\brack{q_h, p_1}_\Gamma + 
       \brack{q_h, p_2}_\Gamma} 
{\sqrt{\epsilon^{-2}\norm{p_1}^2_{\half, \Gamma} + \semi{p_2}^2_{1, \Gamma}}}
\end{split}.
\end{equation}
For each $p_1\in H^{\half}_{00}\open{\Gamma}$ let $u_h \in U_h$ the 
weak solution of the boundary value problem
\[
\begin{aligned}
-\Delta u &= 0 &\mbox{ in }\Omega,\\
\epsilon u & = p_1 &\mbox{ on }\Gamma,\\
         u & = 0 &\mbox{ on }\partial\Omega.
\end{aligned}
\]
Then $\epsilon T_{\Gamma} u_h = p_1$ in $H^{\half}_{00}\open{\Gamma}$
and $\epsilon\semi{u_h}_{1, \Omega}\leq C \norm{p_1}_{\half, \Gamma}$
for some constant $C$ depending only on $\Omega$ and $\Gamma$.
For each $p_2\in
H^1_0\open{\Gamma}$ let $v_h\in V_h$ be the $L^2$ projection of $p_2$ onto the 
space $V_h$
\begin{equation}\label{eq:u2_cstr}
	\brack{v_h - p_2, z}_{\Gamma}=0 \quad z\in V_h.
\end{equation}
By construction we then have $\brack{q_h, p_2-v_h}_{\Gamma}=0$ for all $q_h\in Q_h$ and
  
    $\norm{v_h}_{0, \Gamma}\leq \norm{p_2}_{0, \Gamma}$. Moreover for shape regular
triangulation the projection $\Pi:H^1_0\open{\Gamma}\rightarrow V_h$, $v_h=\Pi p_2$
is bounded in the $H^1_0$ norm
\begin{equation}\label{eq:u2_estim}
  \semi{v_h}_{1, \Gamma} \leq \semi{p_2}_{1, \Gamma}.
\end{equation}
We refer to \cite[ch. 7]{braess} for this result. For constructed $u_h, v_h$
it follows from \eqref{eq:1234} that
\[
\begin{split}
\norm{q_h}_Q &\lesssim
\sup_{w_h \in {U}_h + V_h} \, \sup_{w_h=u_h+v_h}
\frac{\brack{q_h, \epsilon T_{\Gamma} u_h + v_h}_{\Gamma}}
{\sqrt{\semi{u_h}^2_{1, \Omega} + \semi{v_h}^2_{1, \Gamma}}}\\
&=
\sup_{\open{u_h, v_h}\in {U}_h\times V_h}
\frac{\brack{q_h, \epsilon T_{\Gamma} u_h + v_h}_{\Gamma}}{\norm{\open{u_h, v_h}}_{W}}
=
\sup_{w_h\in W_h}
\frac{\brack{B_h w_h, q_h}_{\Gamma}}{\norm{w_h}_{W}}.
\end{split}
\]
\end{proof}
The constructed stable discretizations \eqref{eq:hspaces} are a special case of
conforming spaces built from $U_{h; k}\subset H^1\open{\Omega}$
and $V_{h; l}\subset H^1\open{\Gamma}$ defined as
\begin{equation}\label{eq:high_order}
\begin{aligned}
U_{h; k} &= \set{v\in C\open{\overbar{\Omega}_h}\,:\, v|_{K}=\mathbb{P}_k\open{K}},\\
V_{h; l} &= \set{v\in C\open{\overbar{\Gamma}_h}\,:\, v|_{e}=\mathbb{P}_l\open{e}}.\\
\end{aligned}
\end{equation}

The following corrolary gives a necessary compatibility condition on polynomial 
degrees in order to build inf-sup stable spaces from components \eqref{eq:high_order}.
\begin{corollary}\label{Qcap_infsup_must}
Let $W_{h; k, l}=\open{U_{h;k} \cap H^1_0\open{\Omega}} \times 
\open{V_{h; l} \cap H^1_0\open{\Gamma}}$ and $Q_{h;m}= V_{h; m} \cap H^1_0\open{\Gamma} $. The necessary 
condition for \eqref{eq:discrete_inf_sup} to hold with space 
$W_{h; k,l}\times Q_{h;m}$ is
that $m\leq \max\open{k, l}$.
\end{corollary}
\begin{proof}
Note that $T_{\Gamma} u_h - v_h$ is piecewise polynomial of degree $\max\open{k, l}$.
Suppose $m > \max\open{k, l}$. Then for each $\open{u_h, v_h}\in W_{h;k,l}$ we can
find a orthogonal polynomial $0\neq q_h \in Q_{h;m}$ such that
\[
  \brack{q_h, T_{\Gamma} u_h - v_h}_{\Gamma} = 0.
\]
In turn $\beta=0$ in \eqref{eq:discrete_inf_sup} and the discrete inf-sup
condition cannot hold.
\end{proof}

\subsection{Numerical experiments}\label{sec:Qcap_numerics}
Let now $\AAh$, $\BBQh$ be the matrices \eqref{eq:eigen_matrix},
\eqref{eq:Qcap_matrix} assembled over the constructed stable spaces
\eqref{eq:hspaces}. We demonstrate the robustness of the $Q$-cap preconditioner 
\eqref{eq:Qcap_op} through a pair of numerical experiments.
First, the \textit{exact} preconditioner represented by the matrix $\BBQh$ is considered
and we are interested in the condition number of $\BBQh\AAh$ for different
values of the parameter $\epsilon$. The spectral condition number is computed
from the smallest and largest (in magnitude) eigenvalues of the generalized 
eigenvalue problem $\AAh \mat{x}=\lambda \BBQh^{\Scale[0.5]{-1}}\mat{x}$, which is here
solved by SLEPc
\footnote{We use generalized Davidson method with Cholesky preconditioner and 
convergence tolerance $10^{-8}$.} \cite{slepc}. 
The obtained results are reported in Table \ref{tab:eigs_Qcap}.
In general, the condition numbers are well-behaved indicating that $\BBQh$ 
defines a parameter robust preconditioner. We note that for $\epsilon \ll 1$ the
spectral condition number is close to $\sfrac{\open{1+\sqrt{5}}}{\open{\sqrt{5}-1}}\approx 2.618$.
In \S \ref{sec:Qcap_schur} this observation is explained by the relation of the 
proposed preconditioner $\BBQh$ and the matrix preconditioner of Murphy et al.
\cite{murphy}.
\begin{table}[ht!]
  \begin{center}
  \caption{
Spectral condition numbers of matrices $\BBQh\AAh$ for the system assembled on 
geometry (a) in Figure \ref{fig:scheme}.
  }
\footnotesize{
\begin{tabular}{ll|ccccccc}
\hline
  \multirow{2}{*}{size} & \multirow{2}{*}{$n_Q$} & \multicolumn{7}{c}{$\log_{10}\epsilon$}\\
  \cline{3-9}
  & & $-3$ & $-2$ & $-1$ & $0$ & $1$ & $2$ & $3$\\
\hline
99    &   9 & 2.655 & 2.969 & 4.786 & 6.979 & 7.328 & 7.357 & 7.360\\
323   &  17 & 2.698 & 3.323 & 5.966 & 7.597 & 7.697 & 7.715 & 7.717\\
1155  &  33 & 2.778 & 3.905 & 7.031 & 7.882 & 7.818 & 7.816 & 7.816\\
4355  &  65 & 2.932 & 4.769 & 7.830 & 8.016 & 7.855 & 7.843 & 7.843\\
16899 & 129 & 3.217 & 5.857 & 8.343 & 8.081 & 7.868 & 7.854 & 7.852\\
66563 & 257 & 3.710 & 6.964 & 8.637 & 8.113 & 7.872 & 7.856 & 7.855\\
\hline
\label{tab:eigs_Qcap}
\end{tabular}
}
\end{center}
\end{table}

In the second experiment, we monitor the number of iterations required for 
convergence of the MinRes method \cite{minres}(the implementation is provided by 
cbc.block \cite{block}) applied to the preconditioned equation 
$\overbar{\mathbb{B}}_Q\AAh \vec{x}= \overbar{\mathbb{B}}_Q\vec{b}$. The
operator $\overbar{\mathbb{B}}_Q$ is an efficient and spectrally equivalent 
approximation of $\BBQh$,
\begin{equation}\label{eq:Qcap_precond_approx}
\overbar{\mathbb{B}}_Q = 
\begin{bmatrix}
\text{AMG}\open{\Amat_{U}} &                           & \\
                               & \text{LU}{\open{\Amat_{V}}} & \\
                               &   &      \mat{N}_Q\\
\end{bmatrix},
\end{equation}
with $\mat{N}_Q$ defined in \eqref{eq:Qcap_inverse_formula}. The iterations are 
started from a random initial vector and as a stopping criterion a condition on 
the magnitude of the $k$-th preconditioned residual $\vec{r}_k$, 
$\transp{\vec{r}_k}\overbar{\mathbb{B}}_Q \vec{r}_k < 10^{-12}$ is used. The 
observed number of iterations is shown in Table \ref{tab:Qcap_iters}. 
Robustness with respect to size of the system and the material parameter is evident 
as the iteration count is bounded for all the considered discretizations and
values of $\epsilon$.
\begin{table}[ht!]
  \begin{center}
  \caption{
    Iteration count for convergence of
 $\overbar{\mathbb{B}}_Q \AAh \vec{x}=\overbar{\mathbb{B}}_Q\vec{b}$ 
 solved with the minimal residual method. The problem is assembled on geometry
    (a) from Figure \ref{fig:scheme}.
}
\footnotesize{
\begin{tabular}{ll|ccccccc}
\hline
  \multirow{2}{*}{size} & \multirow{2}{*}{$n_Q$} & \multicolumn{7}{c}{$\log_{10}\epsilon$}\\
  \cline{3-9}
  & & $-3$ & $-2$ & $-1$ & $0$ & $1$ & $2$ & $3$\\
\hline
66563     &  257 & 20 & 34 & 37 & 32 & 28 & 24 & 21\\
264195    &  513 & 22 & 34 & 34 & 30 & 26 & 24 & 20\\
1052675   & 1025 & 24 & 33 & 32 & 28 & 26 & 22 & 18\\
4202499   & 2049 & 26 & 32 & 30 & 26 & 24 & 20 & 17\\
8398403   & 2897 & 26 & 30 & 30 & 26 & 22 & 19 & 15\\
11075583  & 3327 & 26 & 30 & 30 & 26 & 22 & 19 & 15\\
\hline
  \label{tab:Qcap_iters}
\end{tabular}
}
\end{center}
\end{table}

Comparing Tables \ref{tab:eigs_Qcap} and \ref{tab:Qcap_iters} we observe that the 
$\epsilon$-behavior of the condition number and the iteration counts are different. 
In particular, fewer iterations are required for $\epsilon=10^{3}$ than for 
$\epsilon=10^{-3}$ while the condition number in the former case is larger. Moreover, 
the condition numbers for $\epsilon>1$ are almost identical whereas the iteration 
counts decrease as the parameter grows. We note that these observations should
be viewed in the light of the fact that the convergence of the
minimal residual method in general does not depend solely on the condition number, 
e.g. \cite{liesen_tichy}, and a more detailed knowledge of the eigenvalues is
required to understand the behavior.

Having proved and numerically verified the properties of the $Q$-cap
preconditioner, we shall in the next section link $\BBQh$ to a block diagonal 
matrix preconditioner suggested by Murphy et al. \cite{murphy}. Both matrices are assumed to be
assembled on the spaces \eqref{eq:hspaces} and the main objective of the section is 
to prove spectral equivalence of the two preconditioners.

\subsection{Relation to Schur complement preconditioner}\label{sec:Qcap_schur}
Consider a linear system $\AAh\vec{x}=\vec{b}$ with an indefinite matrix
\eqref{eq:eigen_matrix} which shall be preconditioned by a block diagonal 
matrix 
\begin{equation}\label{eq:Schur_mat_preconditioner}
	\mathbb{B}=\diag\open{{\Amat_{{U}}}, {\Amat_{{V}}},
	{\mat{S}}}^{\Scale[0.5]{-1}},
  \quad
\mat{S}=
\Bmat_{{U}}\inv{\Amat_{{U}}}\transp{\Bmat_{{U}}} +
\Bmat_{{V}}\inv{\Amat_{{V}}}\transp{\Bmat_{{V}}},
\end{equation}
where $\mat{S}$ is the negative Schur complement of $\AAh$.
Following \cite{murphy} the spectrum of $\mathbb{B}\AAh$ consists of three
distinct eigenvalues. In fact $\rho\open{\mathbb{B}\AAh}=\set{1,
\tfrac{1}{2} \pm \tfrac{1}{2}\sqrt{5}}$. A suitable Krylov 
method is thus expected to converge in no more than three iterations. However 
in its presented form $\mathbb{B}$ does not define an efficient preconditioner.
In particular, the cost of setting up the Schur complement comes close to inverting 
the system matrix $\AAh$. Therefore a cheaply computable approximation of
$\mat{S}$ is needed to make the preconditioner practical (see e.g. 
\cite[ch. 10.1]{benziliesen} for an overview of generic methods for constructing 
the approximation). We proceed to show that if spaces \eqref{eq:hspaces} are
used for discretization, the Schur complement is more efficiently approximated 
with the inverse of the matrix $\mat{N}_{Q}$ defined in \eqref{eq:Qcap_inverse_formula}.

Let $W_h, Q_h$ be the spaces $\eqref{eq:hspaces}$. Then the mass matrix
$\mat{M}_{{\overbar{U}Q}}=\mat{M}_{{VQ}}$ (cf. discussion prior to
\eqref{eq:mat_defs}) and the matrix will be referred to as $\mat{M}$. Moreover
let us set $\mat{A}_{{V}}=\Amat$. With these definitions the Schur complement 
of $\AAh$ reads
\begin{equation}\label{eq:murphy_schur}
\mat{S}=
\epsilon^2\mat{M}\mat{T}\inv{\Amat_{{U}}}\transp{\mat{T}}\mat{M} +
\mat{M}\inv{\Amat}\mat{M}.
\end{equation}
Further, note that such matrices $\mat{A}$, $\mat{M}$ are suitable for
constructing the approximation of the ${H}_s$ norm on the
space $Q_h$ by the mapping \eqref{eq:H_def}. In particular, $\mat{A}$ is such 
that $\semi{p}^2_{1, \Gamma}=\transp{\vec{p}}\mat{A}\vec{p}$ with $p\in Q_h$ 
and $\vec{p}\in\reals^{n_{{Q}}}$ its coordinate vector. In turn the inverse of
the matrix $\mat{N}_{Q}$ reads
\begin{equation}\label{eq:murphy_Ninv}
\inv{\mat{N}_{Q}}=
\open{\mat{M}\mat{U}}
\open{\epsilon^{2}\mat{\Lambda}^{\nhalf}+\mat{\Lambda}^{\Scale[0.5]{-1}}}
\transp{\open{\mat{M}\mat{U}}}
=
\epsilon^2\Hmat{-\tfrac{1}{2}} + \Hmat{-1}.
\end{equation}
Recalling that $\Hmat{-1}=\mat{M}\inv{\mat{A}}\mat{M}$ and contrasting 
\eqref{eq:murphy_schur} with \eqref{eq:murphy_Ninv} the matrices differ 
only in the first terms. We shall first show that if the terms are spectrally 
equivalent then so are $\mat{S}$ and $\inv{\mat{N}_Q}$. 
\begin{theorem}\label{thm:spec_eq}
Let $\mat{S}$, $\inv{\mat{N}_Q}$ be the matrices defined respectively in
\eqref{eq:murphy_schur} and \eqref{eq:murphy_Ninv} and let $n_{Q}$ be their size. 
Assume that there exist positive constants $c_1, c_2$ dependent only on 
$\Omega$ and $\Gamma$ such that for every $n_Q>0$ and any 
$\vec{p}\in\reals^{n_{{Q}}}$
\[
  c_1 \transp{\vec{p}} \Hmat{-\tfrac{1}{2}} \vec{p}
  \leq
  \transp{\vec{p}} \mat{M}\mat{T}\inv{\Amat_{{U}}}\transp{\mat{T}}\mat{M} \vec{p}
  \leq
  c_2 \transp{\vec{p}} \Hmat{-\tfrac{1}{2}} \vec{p}.
\]
Then for each $n_Q>0$ matrix $\mat{S}$ is spectrally equivalent with 
$\inv{\mat{N}_Q}$.
\end{theorem}
\begin{proof}
By direct calculation we have
\[
\begin{split}
\transp{\vec{p}} \mat{S} \vec{p}  &= 
\epsilon^2\transp{\vec{p}} \mat{M}\mat{T}\inv{\Amat_{{U}}}\transp{\mat{T}}\mat{M}\vec{p} + 
\transp{\vec{p}} \Hmat{-1} \vec{p}\\
&\leq
c_2 \epsilon^2 \transp{\vec{p}} \Hmat{-\tfrac{1}{2}} \vec{p} +
\transp{\vec{p}} \Hmat{-1} \vec{p}\\
&\leq C_2 \transp{\vec{p}} \inv{\mat{N}_Q} \vec{p}
\end{split}
\]
for $C_2=\sqrt{1+c^2_2}$. The existence of lower bound follows from estimate 
\[
\transp{\vec{p}} \mat{S} \vec{p}
\geq
  c_1\epsilon^2\transp{\vec{p}} \Hmat{-\tfrac{1}{2}} \vec{p} + 
\transp{\vec{p}} \Hmat{-1} \vec{p}\\
\geq 
C_1 \transp{\vec{p}} \inv{\mat{N}_Q} \vec{p}
\]
with $C_1=\min\open{1, c_1}$.\qquad
\end{proof}\\
Spectral equivalence of preconditioners $\BBQh$ and $\mathbb{B}$ now follows
immediately from Theorem \ref{thm:spec_eq}. Note that for $\epsilon \ll 1$ the
term $\Hmat{-1}$ dominates both $\mat{S}$ and $\inv{\mat{N}_Q}$. 
In turn, the spectrum of $\mathbb{B}\AAh$ is expected to approximate well the eigenvalues of 
$\BBQh\AAh$. This is then a qualitative explanation of why the spectral condition 
numbers of $\BBQh\AAh$ observed for $\epsilon=10^{-3}$ in Table \ref{tab:eigs_Qcap}
are close to $\sfrac{\open{1+\sqrt{5}}}{\open{\sqrt{5}-1}}$. It remains to prove 
that the assumption of Theorem \ref{thm:spec_eq} holds.
\begin{lemma}\label{lmm:trace}
There exist constants $c_1, c_2 > 0$ depending only on $\Omega, \Gamma$ such 
that for all $n_{{Q}}>0$ and $p \in \reals^{n_{{Q}}}$
\[
  c_1\transp{\vec{p}}\Hmat{-\tfrac{1}{2}}\vec{p}
  \leq
  \transp{\vec{p}} \Mmat\Tmat\inv{\Amat_{U}}\transp{\Tmat}\Mmat \vec{p} 
  \leq      
  c_2\transp{\vec{p}} \Hmat{-\tfrac{1}{2}}\vec{p}.
\]
\end{lemma}
\begin{proof}
For the sake of readability let $n=n_{{Q}}$ and $m=n_{{U}}$. Since $\Mmat$ 
is symmetric and invertible, 
$\Hmat{-\tfrac{1}{2}}=\Mmat\mat{U}\mat{\Lambda}^{\nhalf}\transp{\mat{U}}\Mmat$
and
$\mat{U}\mat{\Lambda}^{\nhalf}\transp{\mat{U}}={\Hmat{\tfrac{1}{2}}}^{\Scale[0.5]{-1}}$ the
statement is equivalent to 
\begin{equation}
  c_1\transp{\vec{y}} \inv{\Hmat{\tfrac{1}{2}}} \vec{y}
  \leq
  \transp{\vec{y}} \Tmat\inv{\Amat_{{U}}}\transp{\Tmat} \vec{y}
  \leq      
  c_2\transp{\vec{y}} \inv{\Hmat{\tfrac{1}{2}}} \vec{y}
  \quad \text{ for all } y \in \reals^{m}.
\end{equation}
The proof is based on properties of the continuous trace operator $T_\Gamma$.
Recall the trace inequality: There exists a positive constant
$K_2=K_2\left(\Omega,\Gamma\right)$ such that $\norm{T_\Gamma u}_{\half, \Gamma} \leq K_2 \semi{u}_{1, \Omega}$ 
for all $u\in H^1_0\open{\Omega}$. From here it follows 
that the sequence $\set{\lambda^{\text{max}}_m}$, where for each $m$ value
$\lambda^{\text{max}}_m$ is the largest eigenvalue of the
eigenvalue problem 
\begin{equation}\label{eq:spec_eq_1}
\transp{\Tmat}\Hmat{\tfrac{1}{2}}\Tmat\vec{u}=\lambda \Amat_{{U}}\vec{u},
\end{equation}
is bounded from above by $K_2$. Note that the eigenvalue problem can be solved with
nontrivial eigenvalue only for $\vec{u}\in\reals^{n}$ for which there exists some 
$\vec{q}\in\reals^m$ such that $\vec{u}=\transp{\Tmat}\vec{q}$. Consequently the 
eigenvalue problem becomes
$\transp{\Tmat}\Hmat{\tfrac{1}{2}}\vec{q}=\lambda\Amat_U\transp{\mat{T}}\vec{q}$. 
Next, applying the inverse of $\Amat_{U}$ and the trace matrix yields
$\Tmat\ninv{\Amat_{U}}\transp{\Tmat}\Hmat{\tfrac{1}{2}}\vec{q}=\lambda\vec{q}$. 
Finally, setting $\vec{q}=\ninv{\Hmat{\tfrac{1}{2}}}\vec{p}$ yields
\begin{equation}\label{eq:spec_eq_2}
\Tmat\inv{\Amat_{U}}\transp{\Tmat}\vec{p}=\lambda\inv{\Hmat{\tfrac{1}{2}}}\vec{p}.
\end{equation}
Thus the largest eigenvalues of \eqref{eq:spec_eq_1} and \eqref{eq:spec_eq_2}
coincide and in turn $C_2=K_2$. Further \eqref{eq:spec_eq_2} has only positive
eigenvalues and the smallest nonzero eigenvalue of 
\eqref{eq:spec_eq_1} is the smallest eigenvalue $\lambda^\text{{min}}_m$  of 
\eqref{eq:spec_eq_2}. Therefore for all $\vec{y}\in\reals^m$
it holds that 
$\lambda^{\text{min}}_m\transp{\vec{y}} \ninv{\Hmat{\tfrac{1}{2}}} \vec{y} 
\leq \transp{\vec{y}} \Tmat\ninv{\Amat_{U}}\transp{\Tmat} \vec{y}$. But the 
sequence $\seq{\lambda^{\text{min}}_m}$ is bounded from below since the right-inverse
of the trace operator is bounded \cite{trace_inverse}.
\qquad
\end{proof}

The proof of Lemma \ref{lmm:trace} suggests that the constants $c_1$, $c_2$ 
for spectral equivalence are computable as the limit of convergent sequences $\set{\lambda^{\text{min}}_m}$,
$\set{\lambda^{\text{max}}_m}$ consisting of the smallest and largest 
eigenvalues of the generalized eigenvalue problem \eqref{eq:spec_eq_2}.
Convergence of such sequences for the two geometries in Figure \ref{fig:scheme} 
is shown in Figure \ref{fig:trace_bounds}. For the simple geometry (a) the 
sequences converge rather fast and the equivalence constants $c_1, c_2$ are
clearly visible in the figure. Convergence on the more complex geometry (b) is 
slower. 
\begin{figure}[t!]
\centering
\includegraphics[width=0.55\textwidth]{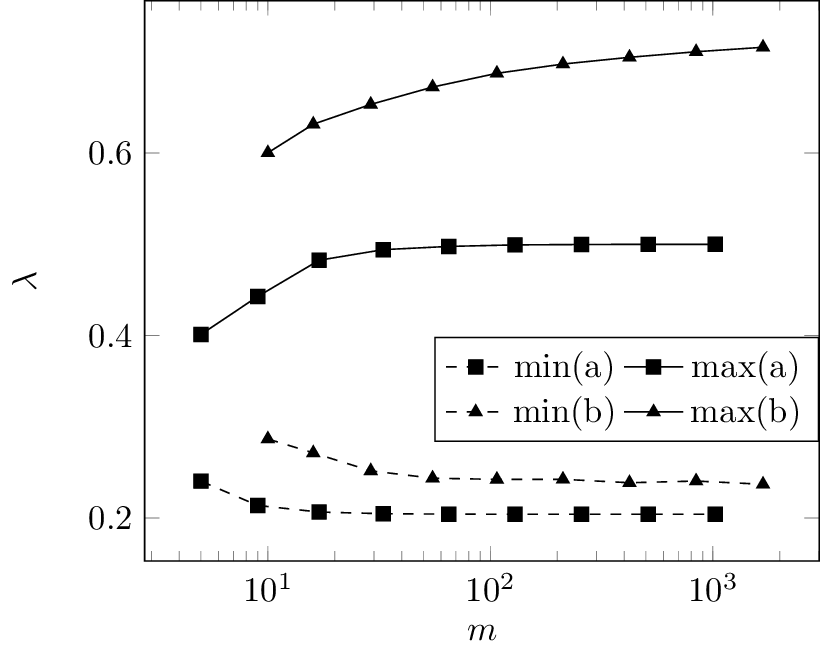}
\caption{Convergence of sequences $\set{\lambda^{\text{max}}_m}$
$\set{\lambda^{\text{min}}_m}$ from Lemma \ref{lmm:trace} for geometries in
Figure \ref{fig:scheme}. For all sequences but $\text{max}\open{b}$ the
constant bound is reached within the considered range of discretization
parameter $m=n_{{Q}}$.}
\label{fig:trace_bounds}
\end{figure}

So far we have by Theorem \ref{thm:stab_eig} and Lemma \ref{Qcap_discrete_inf_sup} 
that the condition numbers of matrices $\BBQh\AAh$ assembled over spaces
\eqref{eq:hspaces} are bounded by constants independent of $\set{h, \epsilon}$. 
A more detailed characterization of the spectrum of the system preconditioned
by the $Q$-cap preconditioner is given next. In particular, we relate the
spectrum to computable bounds $C_1$, $C_2$ and characterize the distribution of
eigenvalues. Further, the effect of varying $\epsilon$ 
(cf. Tables \ref{tab:eigs_Qcap}--\ref{tab:Qcap_iters}) is illustrated by
numerical experiment.

\subsection{Spectrum of the $Q$-cap preconditioned system}\label{sec:Qcap_spectrum}
In the following, the left-right preconditioning of $\AAh$ based on $\BBQh$ is
considered and we are interested in the spectrum of 
\begin{equation}\label{eq:lr}
\BBQh^{\half} \AAh \BBQh^{\half}=
\begin{bmatrix}
\mat{I}_U &            & \Amat_U^{\nhalf}\transp{\mat{B}_U}{\mat{N}_Q}^{\half} \\
          &  \mat{I}_V & \Amat_V^{\nhalf}\transp{\mat{B}_V}{\mat{N}_Q}^{\half} \\
\mat{N}_Q^{\half}\mat{B}_U\Amat^{\nhalf}_U & \mat{N}_Q^{\half}\mat{B}_V\Amat^{\nhalf}_V & \\
\end{bmatrix}.
\end{equation}
The spectra of the left preconditioner system $\BBQh\AAh$ and the left-right 
preconditioned system $\BBQh^{\half} \AAh \BBQh^{\half}$ are identical. 
Using results of \cite{rusten} the
spectrum $\rho$ of \eqref{eq:lr} is such that $\rho=I^{-}\cup I^{+}$ with
\begin{equation}\label{eq:rusten}
I^{-} = \left[\frac{1-\sqrt{1+4\sigma^2_{\text{max}}}}{2},
              \frac{1-\sqrt{1+4\sigma^2_{\text{min}}}}{2}
        \right]
\quad\quad
I^{+} = \left[1, \frac{1+\sqrt{1+4\sigma^2_{\text{max}}}}{2}\right]
\end{equation}
and $\sigma_{\text{min}}, \sigma_{\text{max}}$ the smallest and largest singular
values of the block matrix formed by the first two row blocks in the last column
of $\BBQh^{\half} \AAh \BBQh^{\half}$. We shall denote the matrix as $\mathbb{D}$,
\[
\mathbb{D} =
\begin{bmatrix}
\Amat_U^{\nhalf}\transp{\mat{B}_U}{\mat{N}_Q}^{\half} \\
\Amat_V^{\nhalf}\transp{\mat{B}_V}{\mat{N}_Q}^{\half}
\end{bmatrix}.
\]
\begin{proposition}\label{prep:cond}
The condition number 
$\kappa\open{\BBQh\AAh}$ 
is bounded such that
\[
\kappa\open{\BBQh\AAh} \leq
\frac{1+\sqrt{1 + 4C_2}}{1 - \sqrt{1 + 4C_1}},
\]
where $C_1, C_2$ are the spectral equivalence bounds from Theorem \ref{thm:spec_eq}.
\end{proposition}
\begin{proof}
Note that the singular values of matrix $\mathbb{D}$ and the eigenvalues of matrix
${\mat{N}_Q}^{\half}\mat{S}{\mat{N}_Q}^{\half}$ are identical. Further, using 
Theorem $\ref{thm:spec_eq}$ with $\vec{p}={\mat{N}_Q}^{\half}\vec{q}$,
$\vec{q}\in \reals^{n_{Q}}$ yields
\[
  C_1\transp{\vec{q}}\vec{q}
  \leq
  \transp{\vec{q}} {\mat{N}_Q}^{\half} \mat{S} {\mat{N}_Q}^{\half} \vec{q}
  \leq      
  C_2 \transp{\vec{q}}\vec{q}
  \quad \text{ for all } q \in \reals^{n_{{Q}}}.
\]
In turn the spectrum of matrices ${\mat{N}_Q}^{\half}\mat{S}{\mat{N}_Q^{\half}}$ 
is contained in the interval $\close{C_1, C_2}$. The statement now follows from 
\eqref{eq:rusten}.
\end{proof}
From numerical experiments we observe that the bound due to Proposition \ref{prep:cond} slightly
overestimates the condition number of the system. For example, using numerical
trace bounds (cf. Figure \ref{fig:trace_bounds}) of geometry (a) in Figure
\ref{fig:scheme}, $c_1=0.204, c_2=0.499$ and Theorem \ref{thm:spec_eq}, the formula yields 
$9.607$ as the upper bound on the condition number. On the other hand 
condition numbers reported in Table \ref{tab:eigs_Qcap} do not exceed $8.637$. 
Similarly using estimated bounds for geometry (b) $c_1=0.237, c_2=0.716$ the formula gives 
upper bound $8.676$. The largest condition number in our experiments (not
reported here) was $7.404$.
%
\begin{figure}[t!]
\centering
\includegraphics[width=0.55\textwidth]{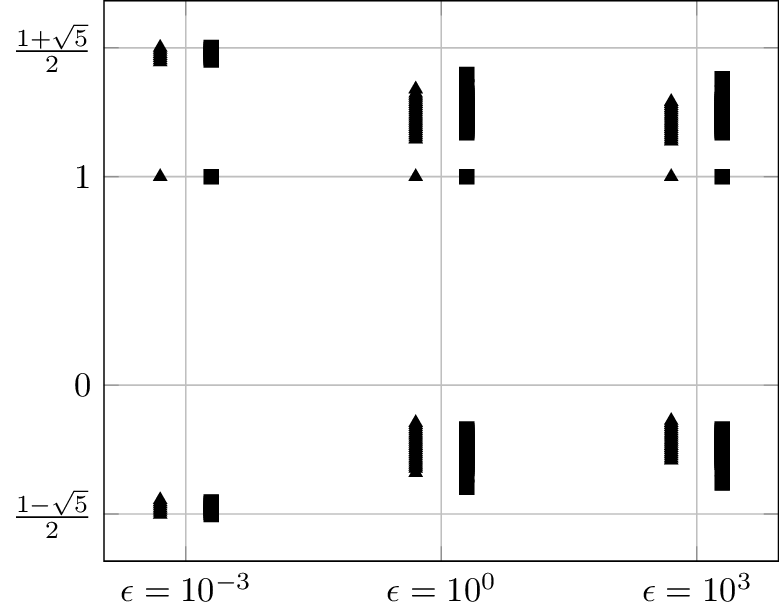}
\caption{Eigenvalues of matrices $\BBQh\AAh$ assembled on
geometries from Figure \ref{fig:scheme} for three different values of
$\epsilon$. The value of $\epsilon$ is indicated by grey vertical lines. On the
left side of the lines is the spectrum for configuration (a). The spectrum for geometry
(b) is then plotted on the right side. For $\epsilon \ll 1$ the eigenvalues cluster 
near $\lambda=1$ and $\lambda=\tfrac{1}{2}\pm\tfrac{1}{2}\sqrt{5}$ (indicated by 
grey horizontal lines) which form the spectrum of $\mathbb{B}\AAh$.}
\label{fig:spectra}
\end{figure}

It is clear that \eqref{eq:rusten} could be used to analyze the effect of the
parameter $\epsilon$ on the spectrum provided that the singular values 
$\sigma_{\text{min}}$, $\sigma_{\text{max}}$ were given as functions of 
$\epsilon$. We do not attempt to give this characterization here. Instead the
effect of $\epsilon$ is illustrated by a numerical experiment. Figure
\ref{fig:spectra} considers the spectrum of $\BBQh\AAh$
assembled on geometries from Figure \ref{fig:scheme} and three different values
of the parameter. The systems from the two geometrical configurations are similar 
in size, 4355 for (a)
and 4493 for (b). Note that for $\epsilon \ll 1$ the eigenvalues for both
configurations cluster near $\lambda=1$ and $\lambda=\tfrac{1}{2}\pm\tfrac{1}{2}\sqrt{5}$, 
that is, near the eigenvalues of $\mathbb{B}\AAh$. This observation is expected
in the light of the discussion following Theorem \ref{thm:spec_eq}. With $\epsilon$ 
increasing the difference between $\BBQh$ and $\mathbb{B}$ caused by 
$\Hmat{-\tfrac{1}{2}}$ becomes visible as the eigenvalues 
are no longer clustered. Observe that in these cases the lengths of intervals 
$I^{-}, I^{+}$ are greater for geometry (b). This observation can be
qualitatively understood via Proposition \ref{prep:cond}, Theorem
\ref{thm:spec_eq} and Figure \ref{fig:trace_bounds} where the trace map constants 
$c_1$, $c_2$ of configuration (a) are more spread than those of (b).

\section{$W$-cap preconditioner}\label{sec:Vcap} 
To circumvent the need for mappings involving fractional Sobolev spaces we
shall next study a different preconditioner for \eqref{eq:weak_form}. As will be 
seen the new preconditioner $W$-cap preconditioner \eqref{eq:operator_preconditionersV} 
is still robust with respect to the material and discretization parameters. 

Consider operator $\op{A}$ from problem \eqref{eq:op_long} as a mapping 
$W\times Q \rightarrow W^*\times Q^*$, with spaces $W, Q$ defined as
\begin{equation}\label{eq:Vcap_spaces}
  \begin{aligned}
    W&=\left(H^1_0\left(\Omega\right)\cap\epsilon\HG\right)\times
    H^1_0\left(\Gamma\right),\\
    Q&= H^{-1}\left(\Gamma\right).
  \end{aligned}
\end{equation}
The spaces are equipped with norms
\begin{equation}\label{eq:Vcap_norms}
  \norm{w}^2_W = \HnormO{u}^2 + \epsilon^2\HnormG{T_{\Gamma} u}^2 + \HnormG{v}^2
  \quad
  \text{ and }\quad
  \norm{p}^2_Q = \norm{p}^2_{-1, \Gamma}.
\end{equation}
Note that the trace of functions from space $U$ is here controlled in the norm
$\HnormG{\cdot}$ and not the fractional norm $\norm{\cdot}_{\half, \Gamma}$ as was 
the case in \S \ref{sec:Qcap}. Also note that the space $W$ now is
dependent on $\epsilon$ while $Q$ is not. The following result establishes well 
posedness of \eqref{eq:weak_form} with the above spaces.
\begin{theorem}\label{thm:stab_inv}
Let $W$ and $Q$ be the spaces \eqref{eq:Vcap_spaces}. The operator
$\op{A}:W\times Q\rightarrow W^*\times Q^*$,
defined in \eqref{eq:op_long} is an isomorphism and the condition number of 
$\op{A}$ is bounded independently of $\epsilon >0$.
\end{theorem}
\begin{proof} The proof proceeds by verifying the Brezzi conditions \ref{thm:brezzi}. 
  With $w=(u, v)$, $\omega=(\phi, \psi)$ application of the Cauchy-Schwarz 
inequality yields
\[
  \begin{split}
    \brack{A w, \omega }_{\Omega} &= \inner{\nabla u }{\nabla \phi}_{\Omega} + 
    \inner{\nabla v}{\nabla \psi}_{\Gamma}
    \\
    & \leq \HnormO{u}\HnormO{\phi} + \HnormG{v}\HnormG{\psi}
    \\
    & \leq \HnormO{u}\HnormO{\phi} + \epsilon^2\HnormG{T_{\Gamma}
    u}\HnormG{\phi} + \HnormG{v}\HnormG{\psi}
    \\
    & \leq \norm{w}_W\norm{\omega}_W.
\end{split}
\]
Therefore $A$ is bounded with $\norm{A}=1$ and \eqref{eq:Brezzi_A_bounded} holds. 
The coercivity of $A$ on $\ker B$ for \eqref{eq:Brezzi_A_coercivity} is obtained 
from
\[
\begin{split}
\inf_{w\in \ker B}\frac{\brack{Aw, w}_{\Omega}}{\norm{w}^2_W} &=
\inf_{w\in \ker B}\frac{\HnormO{u}^2 + \HnormG{v}^2}
{
\HnormO{u}^2+\epsilon^2\HnormG{T_\Gamma u}^2+\HnormG{v}^2
}
\\
&=
\inf_{w\in \ker B}\frac{\HnormO{u}^2+\HnormG{v}^2}
{\HnormO{u}^2 + 2\HnormG{v}^2} \geq \frac{1}{2},
\end{split}
\]
where we used that $\epsilon T_\Gamma u = v$ a.e. on the kernel. Consequently 
$\alpha=\tfrac{1}{2}$. Boundedness of $B$ in \eqref{eq:Brezzi_B_bounded} with a 
constant $\norm{B}=\sqrt{2}$ follows from the Cauchy-Schwarz inequality
\[
\begin{split}
\brack{B w, q}_{\Gamma} &\leq \norm{q}_{-1, \Gamma}\epsilon\HnormG{T_\Gamma
u}+\norm{q}_{-1, \Gamma}\HnormG{v}
       \\
       & \leq
  \sqrt{2}\norm{q}_Q\sqrt{\epsilon^2\HnormG{T_\Gamma u}^2 + \HnormG{v}^2}
       \\
       & \leq
\sqrt{2}\norm{q}_Q\sqrt{\HnormO{u}^2 + \epsilon^2\HnormG{T_{\Gamma} u}^2 + \HnormG{v}^2}
     \\
       & \leq \sqrt{2}\norm{q}_Q\norm{w}_W.
\end{split}
\]
To show that the inf-sup condition holds compute
\[
\begin{split}
\sup_{w\in W}\,\frac{\brack{B w, q}_{\Gamma}}{\norm{w}_W} &= 
\sup_{w\in W}\,\frac{\brack{q, \epsilon T_\Gamma u - v}_{\Gamma}}{\sqrt{
  \HnormO{u}^2 + \epsilon^2 \HnormG{T_\Gamma u}^2 + \HnormG{v}^2}
}
\\
&\stackrel{u=0}\geq
\sup_{v\in V} \frac{\brack{q, v}_{\Gamma}}{\HnormG{v}}=\norm{q}_Q.
\end{split}
\]
Thus $\beta=1$ in condition \eqref{eq:Brezzi_infsup}.
\qquad\end{proof}

Following Theorem \ref{thm:stab_inv} the operator $\mathcal{A}$ is 
a symmetric isomorphism between spaces $W\times Q$ and $W^{*}\times Q^{*}$. As a
preconditioner we shall consider a symmetric positive-definite isomorphism 
$W^{*}\times Q^{*}\rightarrow W\times Q$ %
\begin{equation}\label{eq:Wcap}
\BBW=
\begin{bmatrix}
\inv{\open{-\Delta_{\Omega} + T_{\Gamma}^{*}\open{-\epsilon^2 \Delta_{\Gamma}} T_{\Gamma}}}  &  &\\
			  &  \inv{\open{-\Delta_{\Gamma}}} &  \\
        & & {-\Delta_{\Gamma}}\\
\end{bmatrix}.
\end{equation}
%

%
\subsection{Discrete preconditioner}\label{eq:Vcap_discr} 
Similar to \S \ref{sec:Qcap_discrete} we shall construct discretizations $W_h \times Q_h$ 
of space $W\times Q$ \eqref{eq:Vcap_spaces} such that the finite dimensional 
operator $\mathcal{A}_h$ defined by considering $\mathcal{A}$ from \eqref{eq:op_long} 
on the constructed spaces satisfies the Brezzi conditions \ref{thm:brezzi}.

Let $W_h \subset W$ and $Q_h \subset Q$ the spaces \eqref{eq:hspaces} of
continuous piecewise linear polynomials. Then $A_h$, $B_h$ are continuous with
respect to norms \eqref{eq:Vcap_norms} and it remains to verify conditions 
\eqref{eq:Brezzi_A_bounded} and \eqref{eq:Brezzi_infsup}. First, coercivity of
$A_h$ is considered.
\begin{lemma}\label{Vcap_coerc}
Let $W_h, Q_h$ the spaces \eqref{eq:hspaces} and $A_h, B_h$ such that\\ 
$\brack{A w, \omega_h}_{\Omega}=\brack{A_h w_h, \omega_h}_{\Omega}$,
$\brack{B w, q_h}_{\Gamma}=\brack{B_h w_h, q_h}_{\Gamma}$, for $\omega_h, w_h\in
W_h$, $w\in W$ and $q_h\in Q_h$. Then there exists a constant 
$\alpha > 0$ such that for all $z_h \in \ker B_h$
\[
\brack{A_h z_h, z_h} \geq \alpha \norm{z_h}_{W},
\]
where $\norm{\cdot}_{W}$ is defined in \eqref{eq:Vcap_norms}.
\end{lemma}
\begin{proof}
The claim follows from coercivity of $A$ over $\ker B$ (cf. Theorem
\ref{thm:stab_inv}) and the property $\ker B_h \subset \ker B$. To see that the
inclusion holds, let $z_h \in \ker B_h$. Since $z_h$ is continuous on $\Gamma$
we have from definition $\brack{z_h, q_h}_{\Gamma}=0$ for all $q_h \in Q_h$ that
$z_h|_\Gamma=0$. But then $\brack{z_h, q}=0$ for all $q\in Q$ and therefore $z_h
\in \ker B$.\qquad
\end{proof}
%

Finally, to show that the discretization $W_h \times Q_h$ is stable we show that 
the inf-sup condition for $B_h$ holds.
\begin{lemma}\label{Vcap_infsup}
Let spaces $W_h, Q_h$ and operator $B_h$ from Lemma
\ref{Vcap_coerc}. Then there exists $\beta > 0$ such that
\begin{equation}\label{eq:discrete_inf_sup0}
    \inf_{q_h\in Q_h}\,\sup_{w_h\in W_h}\, 
    \frac{\brack{B_h w_h, q_h}_{\Gamma}}{\norm{w_h}_W\norm{q_h}_Q} \geq \beta,
\end{equation}
where $\norm{\cdot}_{Q}$ is defined in \eqref{eq:Vcap_norms}.
\end{lemma}
\begin{proof}
We first proceed as in the proof of Theorem \ref{thm:stab_inv} and compute
\begin{equation}\label{eq:estim0}
\sup_{w_h \in W_h}
\frac{\brack{q_h, \epsilon T_{\Gamma} u_h - v_h}_{\Gamma}}
{\norm{w_h}_{W}}
\stackrel{u_h=0}{\geq}
\sup_{v_h \in V_h}
\frac{\brack{v_h, q_h}_{\Gamma}}
{\semi{v_h}_{1, \Gamma}}.
\end{equation}
Next, for each $p\in H^1_{0}\open{\Gamma}$ let $v_h=\Pi p$ the element of $V_h$
defined in the proof of Lemma \ref{Qcap_discrete_inf_sup}. In particular, it holds
that
\[
\brack{{p - v_h}, q_h}_{\Gamma} = 0,\quad q_h\in Q_h
\]
and $\semi{v_h}_{1, \Gamma}\leq C \semi{p}_{1, \Gamma}$ for some constant $C$
depending only on $\Omega$ and $\Gamma$. Then
\[
\norm{q_h}_{-1, \Gamma} =
\sup_{p\in H^1_0\open{\Gamma}}
\frac{\brack{q_h, p}_{\Gamma}}
{\semi{p}_{1, \Gamma}}
\leq C
\sup_{v_h\in V_h}
\frac{\brack{q_h, v_h}_{\Gamma}}
{\semi{v_h}_{1, \Gamma}}.
\]
The estimate together with \eqref{eq:estim0} proves the claim of the lemma.
\qquad
\end{proof}

Let now $\mat{A}_{U}, \mat{A}_{V}$ and $\mat{B}_{U}, \mat{B}_{V}$ the matrices 
defined in \eqref{eq:mat_defs} as representations of the corresponding finite
dimensional operators in the basis of the stable spaces $W_h$ and $Q_h$. 
We shall represent the preconditioner $\BBW$ by a matrix
\begin{equation}\label{eq:Vcap_precond}
\BBWh =
\begin{bmatrix}
\inv{\open{\Amat_{U} + \epsilon^{2}\transp{\Tmat}\Amat\Tmat}} & & \\
& \inv{\open{\Amat_{V}}} & \\
& & \inv{\Hmat{-1}}\\
\end{bmatrix},
\end{equation}
where $\inv{\Hmat{-1}}=\inv{\mat{M}}\mat{A}\inv{\mat{M}}$, cf.
\eqref{eq:H_def}, and $\mat{M}$, $\mat{A}$ the matrices inducing $L^2$ and $H^1_0$ inner
products on $Q_h$. Let us point out that there is an obvious correspondence between 
the matrix preconditioner $\BBWh$ and the operator $\BBW$ defined in
\eqref{eq:operator_preconditionersV}. On the other hand it is not entirely
straight forward that the matrix $\BBWh$ represents the $W$-cap preconditioner defined here in 
\eqref{eq:Wcap}. In particular, since the isomorphism from $Q^{*}=H^1_0(\Gamma)$
to $Q=H^{-1}(\Gamma)$ is realized by the Laplacian a case could be made for using the 
stiffness matrix $\mat{A}$ as a suitable representation of the operator.

Let us first argue for $\mat{A}$ not being a suitable representation for
preconditioning. Note that the role of matrix $\mathbb{A}\in\reals^{m\times n}$ in a linear
system $\mathbb{A}\vec{x}=\vec{b}$ is to transform vectors from the solution
space $\reals^n$ to the residual space $\reals^m$. In case the matrix is invertible 
the spaces concide. However, to emphasize the
conceptual difference between the spaces, let us write
$\mathbb{A}:\reals^{n}\rightarrow\reals^{n*}$. Then a preconditioner matrix is a mapping
$\mathbb{B}:\reals^{n*}\rightarrow\reals^{n}$. The stiffness matrix
$\mat{A}$, however, is such that $\mat{A}:\reals^{n_Q}\rightarrow\reals^{n_Q*}$.

It remains to show that $\inv{\mat{M}}\mat{A}\inv{\mat{M}}$ is the correct representation
of $A=-\Delta_{\Gamma}$. Recall that $Q_h\subset Q^*$ and $\mat{A}$ is the matrix
representation of operator $A_h: Q_h\rightarrow Q^{*}_h$. Further, mappings
$\pi_h: Q_h\rightarrow \reals^{n_Q}$, $\mu_h: Q^{*}_h\rightarrow \reals^{n_Q*}$
\[
p_h = \sum_j(\pi_h p_h)_j \chi_j, \quad p_h \in Q_h
\quad\mbox{ and }\quad
(\mu_h f_h)_j = \brack{f_j, \chi_j}, \quad f_h\in Q^{*}_h
\]
define isomorphisms between\footnotemark spaces $Q_h$, $\reals^{n_Q}$ and $Q^*_h$, $\reals^{n_Q*}$ 
\footnotetext{Note that in \S \ref{sec:intro} the mapping $\mu_h$ was considered
as $\mu_h:Q^{*}_h\rightarrow \reals^{n_Q}$. The definition used here reflects
the conceptual distinction between spaces $\reals^{n_Q}$ and $\reals^{n_Q*}$.
That is, $\mu_h$ is viewed as a map from the space of right-hand sides of the
operator equation $A_h p_h = L_h$ to the space of right-hand sides of the
corresponding matrix equation $\mat{A}\vec{p}=\vec{b}$.
}
respectively. We can uniquely associate each $p_h\in Q_h$ with a functional in 
$Q^{*}_h$ via the Riesz map $I_h:Q_h\rightarrow Q^{*}_h$ defined as 
$\brack{I_h p_h, q_h}_{\Gamma}=\inner{p_h}{q_h}_{\Gamma}$. Since
\[
\open{\mu_h I_h p_h}_j = 
\open{I_h p_h, \chi_j}_{\Gamma} = 
\sum_i \open{\pi_h p_h}_i \inner{\chi_i}{\chi_j}_{\Gamma}
\]
the operator $I_h$ is represented as the mass matrix $\mat{M}$. The matrix then
provides a natural isomorphism from $\reals^{n_Q}$ to $\reals^{n_Q*}$. In turn
$\inv{\mat{M}}\mat{A}\inv{\mat{M}}:\reals^{n_Q*}\rightarrow\reals^{n_Q}$ has the 
desired mapping properties. In conclusion, the inverse of the mass matrix
was used in \eqref{eq:Vcap_precond} as a natural adapter to obtain a matrix operating 
between spaces suitable for preconditioning.

Finally, we make a few observations about the matrix preconditioner $\BBWh$. 
Recall that the $Q$-cap preconditioner $\BBQh$ could be related to the Schur 
complement based preconditioner \eqref{eq:Schur_mat_preconditioner} obtained 
by factorizing $\AAh$ in \eqref{eq:eigen_matrix}. The relation of 
$\AAh$ to the $W$-cap preconditioner matrix \eqref{eq:Vcap_precond} is 
revealed in the following calculation
\begin{equation}\label{eq:aux}
  \mathbb{U} \mathbb{L} \AAh = 
\begin{bmatrix}
\Amat_{V} + \epsilon^{2}\transp{\Tmat}\Amat_{}\Tmat & & \\
                                   & \tau^{2}\Amat_{} & -\Mmat_{}\\
-\epsilon\Mmat_{}\Tmat&  & \Mmat_{}\inv{\Amat_{}}\Mmat_{}\\
\end{bmatrix},
\end{equation}
where
\[
\mathbb{U} = 
\begin{bmatrix}
  \mat{I} & \phantom{\transp{\Tmat}\Amat_{}\inv{\Mmat_{}}} &
  -\transp{\Tmat}\epsilon\Amat_{}\inv{\Mmat_{}}\\
        & \mat{I} & \\
        &         & \mat{I}\\
\end{bmatrix}\,
\quad\text{ and }\quad
\mathbb{L} =
\begin{bmatrix}
\mat{I} & &\\
        & \mat{I} & \\
        & -\Mmat_{}\inv{\mat{A}_{}} & -\mat{I}\\
\end{bmatrix}.
\]
Here the matrix $\mathbb{L}$ introduces a Schur complement of a submatrix of
$\AAh$ corresponding to spaces $V_h, Q_h$.
The matrix $\mathbb{U}$ then eliminates the 
constraint on the space $U_h$. Preconditioner $\BBWh$ could now 
be interpreted as coming from the diagonal of the resulting matrix in
\eqref{eq:aux}. Futher, note that the action of the $Q_h$-block can be computed 
cheaply by Jacobi iterations with a diagonally preconditioned mass matrix 
(cf. \cite{wathen_M}).

\begin{table}[ht!]
  \begin{center}
  \caption{Spectral condition numbers of matrices ${\BBWh\AAh}$ 
  for the system assembled on geometry (a) in Figure \ref{fig:scheme}.
  }
  \footnotesize{
\begin{tabular}{l|ccccccc}
\hline
  \multirow{2}{*}{size} & \multicolumn{7}{c}{$\log_{10}\epsilon$}\\
  \cline{2-8}
  & $-3$ & $-2$ & $-1$ & $0$ & $1$ & $2$ & $3$\\
\hline
99    & 2.619 & 2.627 & 2.546 & 3.615 & 3.998 & 4.044 & 4.048\\
323   & 2.623 & 2.653 & 2.780 & 3.813 & 4.023 & 4.046 & 4.049\\
1155  & 2.631 & 2.692 & 3.194 & 3.925 & 4.036 & 4.048 & 4.049\\
4355  & 2.644 & 2.740 & 3.533 & 3.986 & 4.042 & 4.048 & 4.049\\
16899 & 2.668 & 2.788 & 3.761 & 4.017 & 4.046 & 4.049 & 4.049\\
66563 & 2.703 & 3.066 & 3.896 & 4.033 & 4.047 & 4.049 & 4.049\\
\hline
\end{tabular}
\label{tab:eigs_Vcap}
}
\end{center}
\end{table}

\subsection{Numerical experiments}
Parameter robust properties of the $W$-cap preconditioner are demonstrated by 
the two numerical experiments used to validate the $Q$-cap preconditioner in 
\S \ref{sec:Qcap_numerics}. Both the experiments use discretization of
domain (a) from Figure \ref{fig:scheme}. First, using the \textit{exact} preconditioner 
we consider the spectral condition numbers of matrices ${\BBWh \AAh}$. Next, using 
an approximation of $\BBWh$ the linear system $\overbar{\mathbb{B}}_W \AAh \vec{x}=\overbar{\mathbb{B}}_W\vec{f}$ 
is solved with the minimal residual method. The operator $\overbar{\mathbb{B}}_W$ is 
defined as
\begin{equation}\label{eq:B2h_approx}
\overbar{\mathbb{B}}_W = 
\begin{bmatrix}
\text{AMG}{\open{\Amat_U + \epsilon^{2}\transp{\Tmat}\Amat\Tmat}} & & \\
 & \text{LU}{\open{\Amat}} & \\
 & & \text{LU}\open{\Mmat}\,\Amat\,\text{LU}\open{\Mmat}\\
\end{bmatrix}.
\end{equation}
%

The spectral condition numbers of matrices ${\BBWh \AAh}$ for different values of 
material parameter $\epsilon$ are listed in Table \ref{tab:eigs_Vcap}. For all the 
considered discretizations the condition numbers are bounded with respect to $\epsilon$. 
We note that the mesh convergence of the condition numbers appears to be faster
and the obtained values are in general smaller than in case of the $Q$-cap 
preconditioner (cf. Table \ref{tab:eigs_Qcap}).

Table \ref{tab:Vcap_iters} reports the number of iterations required for 
convergence of the minimal residual method for the linear system
$\overbar{\mathbb{B}}_W \AAh \vec{x}=\overbar{\mathbb{B}}_W\vec{f}$. Like for the $Q$-cap preconditioner 
the method is started from a random initial vector and the condition
$\transp{\vec{r}_k}\overbar{\mathbb{B}}_W \vec{r}_k < 10^{-12}$ is used as 
a stopping criterion. We find that the iteration counts with the $W$-cap
preconditioner are again bounded for all the values of the parameter $\epsilon$.
Consistent with the observations about the spectral condition number, the
iteration count is in general smaller than for the system preconditioned with
the $Q$-cap preconditioner.

\begin{table}[ht!]
  \begin{center}
  \caption{Iteration count for system
 $\overbar{\mathbb{B}}_W\AAh \vec{x}=\overbar{\mathbb{B}}_W\vec{f}$ 
 solved with the minimal residual method. The problem is assembled on geometry
 (a) from Figure \ref{fig:scheme}. Comparison to the number of iterations with the 
$Q$-cap preconditioned system is shown in the brackets (cf. also Table \ref{tab:Qcap_iters}).
 }
 \footnotesize{
\begin{tabular}{l|ccccccc}
\hline
  \multirow{2}{*}{size} & \multicolumn{7}{c}{$\log_{10}\epsilon$}\\
  \cline{2-8}
  & $-3$ & $-2$ & $-1$ & $0$ & $1$ & $2$ & $3$\\
\hline
66563    & 17(-3) & 33(-1)& 40(3) & 30(-2) & 20(-8) & 14(-10) & 12(-9)\\
264195   & 19(-3) & 35(1) & 39(5) & 28(-2) & 19(-7) & 14(-10) & 11(-9)\\
1052675  & 22(-2) & 34(1) & 37(5) & 27(-1) & 19(-7) & 14(-8)  & 11(-7)\\
4202499  & 24(-2) & 34(2) & 34(4) & 25(-1) & 17(-7) & 12(-8)  & 9(-8)\\
8398403  & 25(-1) & 32(2) & 32(2) & 24(-2) & 16(-6) & 11(-8)  & 8(-7)\\
11075583 & 25(-1) & 32(2) & 32(2) & 25(-1) & 16(-6) & 13(-6)  & 11(-4)\\
\hline
\end{tabular}
\label{tab:Vcap_iters}
}
\end{center}
\end{table}

We note that the observations from \S \ref{sec:Qcap_numerics} about the difference 
in $\epsilon$-dependence of condition numbers and iteration counts of the 
$Q$-cap preconditioner apply to the $W$-cap preconditioner as well.

Before addressing the question of computational costs of the proposed
preconditioners let us remark that the $Q$-cap preconditioner and the $W$-cap 
preconditioners are not spectrally equivalent. Further, both preconditioners yield
numerical solutions with linearly(optimaly) converging error, see Appendix 
\ref{sec:appendix_eoc}.


%
\section{Computational costs}\label{sec:cpu_cost} 
We conclude by assessing computational efficiency of the proposed preconditioners. 
In particular, the setup cost and its relation to the aggregate solution time of
the Krylov method is of interest. For simplicity we let $\epsilon=1$.

In case of the $Q$-cap preconditioner discretized as \eqref{eq:Qcap_precond_approx}
the setup cost is determined by the construction of algebraic multigrid (AMG) and 
the solution of the generalized eigenvalue problem $\mat{A}\mat{x}=\lambda \mat{M}\mat{x}$ (GEVP). 
The problem is here solved by calling OpenBLAS\cite{openblas} implementation of 
LAPACK\cite{lapack} routine DSYGVD. The setup cost of the $W$-cap preconditioner is 
dominated by the construction of 
multigrid for operator $\Amat_U + \transp{\Tmat}\Amat\Tmat$. We found that
the operator can be assembled with negligible costs and therefore do not report
timings of this operation.

The setup costs of the preconditioners obtained on a Linux machine with 16GB RAM
and Intel Core i5-2500 CPU clocking at 3.3 GHz are reported in Table 
\ref{tab:timings_unif}. We remark that timings on the finest discretization deviate 
from the trend set by the predecessors. This is due to SWAP memory being required 
to complete the operations and the case should therefore be omitted from the 
discussion. On the remaining discretizations the following observations can be 
made: (i) the solution time always dominates the construction time by a factor 
5.5 for $W$-cap and 3.5 for $Q$-cap, (ii) $W$-cap preconditioner is close to two 
times cheaper to construct than the $Q$-cap preconditioner in the form 
\eqref{eq:Qcap_precond_approx}, (iii) the eigenvalue problem always takes fewer 
seconds to solve than the construction of multigrid.

For our problems of about 11 million nodes in the $2d$ domain, the strategy of 
solving the generalized eigenvalue problem using a standard LAPACK routine 
provided an adequate solution. However, the DSYGVD routine appears to be nearly cubic in 
complexity ($\mathcal{O}(n^{2.70}_Q)$ or $\mathcal{O}(n^{1.35}_U)$, cf. Table
\ref{tab:timings_unif}), which may represent a bottleneck for larger problems. 
However, the transformation $\mat{M}_l^{\nhalf}\mat{A} \mat{M}_l^{\nhalf}$ with 
$\mat{M}_l$ the lumped mass matrix presents a simple trick providing significant
speed-up. In fact, the resulting eigenvalue problem is symmetric
and tridiagonal
and can be solved with fast algorithms of nearly quadratic complexity
\cite{demmel, mrrr}. Note that the tridiagonal property holds under the 
assumption of $\Gamma$ having no bifurcations and that the elements are linear.
To illustrate the potential gains with mass lumping, 
using the transformation and applying the dedicated LAPACK routine DSTEGR we were 
able to compute eigenpairs for systems of order sixteen thousand in about fifty 
seconds. This presents more than a factor ten speed up relative to the original 
generalized eigenvalue problem. The value should also be viewed in the light of 
the fact that the relevant space $U_h$ has in this case about quarter billion 
degrees of freedom. We remark that \cite{gevp_s3d} presents a method for 
computing all the eigenpairs of the generalized symmetric tridiagonal eigenvalue 
problem with an estimated quadratic complexity.

Let us briefly mention a few alternative methods for realizing the mapping between
fractional Sobolev spaces needed by the $Q$-cap preconditioner. The methods have
a common feature of computing the action of operators rather than constructing
the operators themselves. Taking advantage of the fact that $\Hmat{s}=\mat{M}\mat{S}^{-s}$,
$\mat{S}=\inv{\mat{A}}\mat{M}$, the action of the powers of the matrix $\mat{S}$ is 
efficiently computable by contour integrals \cite{contour}, symmetric Lanczos 
process \cite{arioli_ima, arioli_siam} or, in case the matrices $\mat{A}$, $\mat{M}$ 
are structured, by fast Fourier transform \cite{fft}. Alternatively, the mapping 
can be realized by the BPX preconditioner \cite{bramble_BPX, bramble_scales} or
integral operator based preconditioners, e.g. \cite{steinbach_bdry}. The above
mentioned techniques are all less than $\mathcal{O}(n_Q^2)$ in complexity.

In summary, for linear elements and geometrical configurations 
where $\Gamma$ is free of bifurcations the eigenvalue problem required for \eqref{eq:H_def} 
lends itself to solution methods with complexity nearing that of the multigrid construction. 
In such case the $Q$-cap preconditioner \eqref{eq:Qcap_precond_approx} is
feasible whenever the methods deliver acceptable performance 
($n_Q\sim 10^4$). For larger spaces $Q_h$ a practical realization
of the $Q$-cap preconditioner could be achieved by one of the listed
alternatives.

\begin{table}[h]
\begin{center}
  \caption{Timings of elements of construction of the $Q$, $W$-cap for
  $\epsilon=1$ and discretizations from Table \ref{tab:Qcap_iters}, 
  \ref{tab:Vcap_iters}. Estimated complexity of computing quantity $v$ at $i$-th row,
  $r_i=\sfrac{\log{v_i}-\log{v_{i-1}}}{\log{m_i}-\log{m_{i-1}}}$ is shown in the
  brackets. Fitted complexity of computing $v$, $\mathcal{O}({n^r_Q})$ is obtained by 
  least-squares. All fits but GEVP ignore the SWAP effected final discretization. 
}
\label{tab:timings_unif}
\footnotesize{
\begin{tabular}{ll|lll|ll}
\hline
\multirow{2}{*}{$n_U$}  & \multirow{2}{*}{$n_Q$}  & \multicolumn{3}{c|}{$Q$-cap} & \multicolumn{2}{c}{$W$-cap}\\
  \cline{3-7}
  & & AMG$\left[s\right]$ & GEVP$\left[s\right]$ & MinRes$\left[s\right]$ &
      AMG$\left[s\right]$ & MinRes$\left[s\right]$ \\
\hline
   66049 & 257  & 0.075(1.98)  & 0.014(1.81)  &  0.579(1.69)  & 0.078(1.94)  & 0.514(1.73)  \\
  263169 & 513  & 0.299(2.01)  & 0.066(2.27)  &  2.286(1.99)  & 0.309(1.99)  & 2.019(1.98)  \\
 1050625 & 1025 & 1.201(2.01)  & 0.477(2.87)  &  8.032(1.82)  & 1.228(1.99)  & 7.909(1.97)  \\
 4198401 & 2049 & 4.983(2.05)  & 3.311(2.80)  &  30.81(1.94)  & 4.930(2.01)  & 30.31(1.94) \\
 8392609 & 2897 & 9.686(1.92)  & 8.384(2.68)  &  62.67(2.05)  & 10.64(2.22)  & 59.13(1.93) \\
11068929 & 3327 & 15.94(3.60)  & 12.25(2.74)  &  84.43(2.15)  & 15.65(2.79)  & 82.13(2.37) \\
\hline
  \multicolumn{2}{l|}{Fitted complexity}
  & \multicolumn{1}{c}{(2.02)} & \multicolumn{1}{c}{(2.70)} & \multicolumn{1}{c|}{(1.92)} &
     \multicolumn{1}{c}{(2.02)} & \multicolumn{1}{c}{(1.96)} \\
\hline
\end{tabular}
}
\end{center}
\end{table}

\section{Conclusions}
We have studied preconditioning of model multiphysics problem \eqref{pde} with
$\Gamma$ being the subdomain of $\Omega$ having codimension one. Using operator
preconditioning \cite{kent_ragnar} two robust preconditioners were proposed and 
analyzed. Theoretical findings obtained in the present treatise about robustness of 
preconditioners with respect to material and discretization parameter were demonstrated 
by numerical experiments using a stable finite element approximation for the related 
saddle point problem developed herein. 
Computational efficiency of the preconditioners was assessed revealing that the $W$-cap 
preconditioner is more practical. The $Q$-cap preconditioner with discretization
based on eigenvalue factorization is efficient for smaller problems and its
application to large scale computing possibly requires different means of realizing 
the mapping between the fractional Sobolev spaces.

Possible future work based on the presented ideas includes extending the
preconditioners to problems coupling $3d$ and $1d$ domains, problems with 
multiple disjoint subdomains and problems describing different physics on the
coupled domains. In addition, a finite element discretization of the
problem, which avoids the constraint for $\Gamma_h$ to be aligned with facets of
$\Omega_h$ is of general interest.

\section*{Acknowledgments}
We would like to thank the anonymous referees for their valuable and
constructive comments, which improved the presention of this paper.

\appendix
\section{Brezzi theory}\label{sec:appendix_brezzi}

\begin{theorem}[Brezzi]\label{thm:brezzi}
The operator $\mathcal{A}:V\times Q \rightarrow V^*\times Q^*$ in
\eqref{eq:op_short} is an isomorphism if the following conditions are satisfied
\begin{subequations}\label{eq:Brezzi_conditions}
\begin{enumerate}[(a)]
  \item $A$ is bounded, 
    \begin{equation}
      \sup_{u\in V}\,\sup_{v\in V} \frac{\brack{A u,
      v}}{\norm{u}_V\norm{v}_V}=c_A \equiv \norm{A} < \infty,
\label{eq:Brezzi_A_bounded}
    \end{equation}
  \item $A$ is invertible on $\ker B$, with 
    \begin{equation}\label{eq:Brezzi_A_coercivity}
      \inf_{u\in \ker B} \frac{\brack{A u, u}}{\norm{u}^2_V} \geq \alpha > 0
    \end{equation}
  \item $B$ is bounded, 
    \begin{equation}\label{eq:Brezzi_B_bounded}
      \sup_{q\in Q}\,\sup_{v\in V}\, \frac{\brack{Bv, q}}{\norm{v}_V\norm{q}_Q} =
      c_B \equiv \norm{B} < \infty,
    \end{equation}
  \item $B$ is surjective (also inf-sup or LBB condition), with
    \begin{equation}\label{eq:Brezzi_infsup}
      \inf_{q\in Q}\,\sup_{v\in V}\, \frac{\brack{Bv, q}}{\norm{v}_V\norm{q}_Q} \geq
      \beta > 0.
    \end{equation}
\end{enumerate}
\end{subequations}
The operator norms $\norm{\mathcal{A}}$ and $\norm{\mathcal{A}^{-1}}$ are bounded in 
terms of the constants appearing in (a)-(d).
\end{theorem}
\begin{proof}
See for example \cite{BrezziFortin}. \qquad
\end{proof}

\section{Estimated order of convergence}\label{sec:appendix_eoc}
Refinements of a uniform discretization of geometry (a) in Figure \ref{fig:scheme}
are used to establish order of convergence of numerical solutions of a manufactured problem 
obtained using $Q$-cap and $W$-cap preconditioners. The error of discrete solutions 
$u_h$ and $v_h$ is interpolated by discontinuous piecewise cubic polynomials and 
measured in the $H^1_0$ norm. The observed convergence rate is
linear(optimal).\\
\begin{center}
\footnotesize{
  \begin{tabular}{l|c c | c c }
  \hline
    \multirow{2}{*}{size}  & \multicolumn{2}{c|}{$Q$-cap} & \multicolumn{2}{c}{$W$-cap}\\
    \cline{2-5}
    & 
$\semi{u-u_h}_{1, \Omega}$ & $\semi{v-v_h}_{1, \Gamma}$  &
$\semi{u-u_h}_{1, \Omega}$ & $\semi{v-v_h}_{1, \Gamma}$ \\
  \hline
16899   & $3.76\times{10}^{-2}$(1.00) & $1.32\times{10}^{-2}$(1.00) & $3.76\times{10}^{-2}$(1.00) & $1.32\times{10}^{-2}$(1.00)\\
66563   & $1.88\times{10}^{-2}$(1.00) & $6.58\times{10}^{-3}$(1.00) & $1.88\times{10}^{-2}$(1.00) & $6.58\times{10}^{-3}$(1.00)\\
264195  & $9.39\times{10}^{-3}$(1.00) & $3.29\times{10}^{-3}$(1.00) & $9.39\times{10}^{-3}$(1.00) & $3.29\times{10}^{-3}$(1.00)\\
1052675 & $4.70\times{10}^{-3}$(1.00) & $1.64\times{10}^{-3}$(1.00) & $4.70\times{10}^{-3}$(1.00) & $1.64\times{10}^{-3}$(1.00)\\
4202499 & $2.35\times{10}^{-3}$(1.00) & $8.22\times{10}^{-4}$(1.00) & $2.35\times{10}^{-3}$(1.00) & $8.22\times{10}^{-4}$(1.00)\\
  \hline
  \end{tabular}
}
\end{center}

\bibliographystyle{siam}
\bibliography{paper}
\end{document}